%% file: IsomProb.tex
\documentclass[12pt,oneside,english]{amsart}
\usepackage[T1]{fontenc}
\usepackage[utf8]{inputenc}
\usepackage{geometry}
\usepackage{babel}
\usepackage{verbatim}
\usepackage{varioref}
\usepackage{amstext}
\usepackage{amsthm}
\usepackage{amssymb}

\makeatletter
\numberwithin{equation}{section}
\numberwithin{figure}{section}
\theoremstyle{plain}
\newtheorem{thm}{\protect\theoremname}[section]
\theoremstyle{remark}
\newtheorem{rem}[thm]{\protect\remarkname}
\theoremstyle{plain}
\newtheorem{prop}[thm]{\protect\propositionname}
\theoremstyle{plain}
\newtheorem{lem}[thm]{\protect\lemmaname}
\theoremstyle{definition}
\newtheorem{defn}[thm]{\protect\definitionname}
\theoremstyle{plain}
\newtheorem{cor}[thm]{\protect\corollaryname}
\theoremstyle{definition}
\newtheorem{problem}[thm]{\protect\problemname}
\theoremstyle{plain}
\newtheorem{conjecture}[thm]{\protect\conjecturename}

\usepackage{hyperref}

\makeatother

\providecommand{\conjecturename}{Conjecture}
\providecommand{\corollaryname}{Corollary}
\providecommand{\definitionname}{Definition}
\providecommand{\lemmaname}{Lemma}
\providecommand{\problemname}{Problem}
\providecommand{\propositionname}{Proposition}
\providecommand{\remarkname}{Remark}
\providecommand{\theoremname}{Theorem}

\begin{document}
\title[The isomorphism problem of projective schemes]{The isomorphism problem of projective schemes and related algorithmic
problems}
\author{Takehiko Yasuda}
\address{Department of Mathematics, Graduate School of Science, Osaka University
Toyonaka, Osaka 560-0043, JAPAN}
\email{yasuda.takehiko.sci@osaka-u.ac.jp}
\thanks{This work was supported by JSPS KAKENHI Grant Number JP18H01112.}
\begin{abstract}
We discuss the isomorphism problem of projective schemes; given two
projective schemes, can we algorithmically decide whether they are
isomorphic? We give affirmative answers in the case of one-dimensional
projective schemes, the case of smooth irreducible varieties with
a big canonical sheaf or a big anti-canonical sheaf, and the case
of K3 surfaces with a finite automorphism group. As related algorithmic
problems, we also discuss decidability of positivity properties of
invertible sheaves, and approximation of the nef cone and the pseudo-effective
cone.
\end{abstract}

\maketitle
\input{macros.tex}

\tableofcontents{}

\section{Introduction}

The main purpose of this paper is to discuss the isomorphism problem
of projective schemes over the field $\overline{\QQ}$ of algebraic
numbers. Is it algorithmically decidable whether two given projective
$\overline{\QQ}$-schemes are isomorphic? What if we restrict ourselves
to some class of projective schemes, for example, the class of smooth
projective varieties having a prescribed invariant. Poonen \cite{poonen2011automorphisms}
writes that Totaro asked him about this problem in 2007. The case
of smooth irreducible curves was treated earlier in the 2005 paper
\cite{baker2005finiteness} by Baker, González-Jiménez, González,
and Poonen. The same problem was asked also by Arapura on MathOverflow\footnote{https://mathoverflow.net/questions/21883/isomorphism-problem-for-commutative-algebras-and-schemes}
in 2010. 

To the best of the author's knowledge, the decidability of the isomorphism
problem has been proved in the following two cases:
\begin{enumerate}
\item Smooth irreducible curves (\cite[Lem.\ 5.1]{baker2005finiteness}
for the case of genus $\ne1$ and Poonen's comment in the MathOverflow
thread mentioned above for the case of genus one). 
\item Varieties of general type (see \cite[Rem.\ 12.3]{poonen2014undecidable}
for the proof due to Totaro).
\end{enumerate}
In the same MathOverflow thread as above, there is also discussion
about the cases of K3 surfaces and abelian surfaces, which has not
reached a definite conclusion. 

The main result of the paper is to prove that the isomorphism problem
is decidable in the following cases:
\begin{enumerate}
\item One-dimensional projective schemes (Theorem \ref{thm:1-dim proj}).
\item Smooth projective varieties with a big canonical sheaf or a big anti-canonical
sheaf (Theorem \ref{thm:iso prob big}).
\item K3 surfaces with a finite automorphism group (Theorem \ref{thm: decidable K3}). 
\end{enumerate}
The first two cases slightly generalize ones mentioned above. As an
application of the first case, we show the decidability also in the
case of one-dimensional reduced \emph{quasi-projective} schemes (Theorem
\ref{thm: iso one-dim quasi-proj}). In birational geometry, varieties
of Kodaira dimensions $-\infty$ and $0$ as well as ones of general
type have special importance, as they are considered as building blocks
of all varieties. With case (2) above being solved, varieties of Kodaira
dimension $0$ would be the remaining most imporant case. Besides
K3 surfaces, the isomorphism problem for abelian varieties should
be important, but we do not discuss it in this paper. 
\begin{rem}
It appears difficult to apply the global Torelli theorem for K3 surfaces
\cite[p.\ 332]{barth2004compact} to solve the isomorphism problem
affirmatively. We can approximate the Hodge structure on cohomology
groups with arbitrary precision \cite{simpson2008algebraic}. But
the moduli space of (marked) K3 surfaces is not Hausdorff \cite[p.\ 334]{barth2004compact}.
This suggests that we cannot detect non-isomorphism of K3 surfaces
by approximation. 
\end{rem}

Our strategy to prove these results is to compute the Iso schemes
$\ulIso_{P_{i}}(X,Y)$ for the given projective schemes $X$ and $Y$
and for finitely many polynomials $P_{i}$. The entire Iso scheme
$\ulIso(X,Y)$ is the moduli scheme of isomorphisms $X\xrightarrow{\sim}Y$
and can be embedded into the Hilbert scheme $\Hilb(X\times Y)$ by
sending an isomorphism $f\colon X\to Y$ to its graph $\Gamma_{f}\subset X\times Y$.
The Hilbert scheme is decomposed as $\Hilb(X\times Y)=\coprod_{P}\Hilb_{P}(X\times Y)$,
where $P$ runs over countably many polynomials. This induces a decomposition
$\ulIso(X,Y)=\coprod_{P}\ulIso(X,Y)$ of the Iso scheme. For each
polynomial $P$, $\ulIso_{P}(X,Y)$ is of finite type, but the entire
$\ulIso(X,Y)$ is not generally so. We explain how to algorithmically
compute $\ulIso_{P}(X,Y)$ for each $P$. Having the method of computing
Iso schemes, we then construct an algorithm for each of the classes
of projective schemes mentioned above that produces finitely many
polynomials $P_{1},\dots,P_{n}$ from the given projective schemes
$X$ and $Y$. These polynomials satisfy the condition that $X$ and
$Y$ are isomorphic if and only if $\ulIso_{P_{i}}(X,Y)\ne\emptyset$
for some $i$. Then, whether $X$ and $Y$ are isomorphic or not is
checked by computing these Iso schemes. In construction of finitely
many polynomials as above, we use the Kodaira vanishing as a key ingredient
in case (2) and use computation of the nef cone in case (3). 

We also discuss several algorithmic problems related to the isomorphism
problem. Firstly, we explicitly describe an algorithm to check whether
two given projective schemes embedded in the same projective space
are projectively equivalent (Section \ref{sec:Projective-equaivalence}).
If two projective schemes are projectively equivalent, then they are
isomorphic, but the converse does not generally hold. Secondly, partly
using computation of intersection numbers, we discuss positivity properties
of invertible sheaves from the algorithmic viewpoint. We see that
global generation of a coherent sheaf and very ampleness of an invertible
sheave on a projective scheme is decidable. Using the Nakai-Moishezon
criterion for ampleness and computation of intersection numbers, we
see that, if the scheme is smooth and irreducible, then ampleness
of an invertible sheaf is also decidable (Proposition \ref{prop: ample}).
We do not know whether other positivities, bigness, nefness and pseudo-effectivity,
are decidable. However, if we can compute the Picard number, then
we can approximate the nef cone and the pseudo-effective cone with
arbitrary precision (Proposition \ref{prop: approx cones}). Note
that Poonen, Testa and van Luijk \cite{poonen2015computing} proved
that the Picard number of a smooth irreducible projective variety
is computable, if the Tate conjecture is true. This is the case for
K3 surfaces. For a K3 surface with a finite automorphism group, we
can compute its nef cone (not approximately but exactly), which is
used to show the decidability in case (3) above. 

To end this introduction, we mention a few more related works. Truong
\cite{truong2018bounded} proved the decidability of the bounded birationality
problem. Namely, he proved that for projective varieties $X\subset\PP^{m}$
and $Y\subset\PP^{n}$ and for a positive integer $d$, we can decide
whether there exists a rational map $\PP^{m}\dasharrow\PP^{n}$ of
degree $\le d$ that restricts to a birational map $X\dasharrow Y$.
He also proved the decidability of the \emph{bounded }isomorphism
problem in the case where one of the two given varieties is smooth.
To prove these results, he showed computability of a variety parametrizing
rational maps $\PP^{m}\dasharrow\PP^{n}$ with this property, which
is similar to our computability result regarding Iso schemes.

The isomorphism problem that we consider in this paper is a speical
case of the problem regarding the existence of a morphism $X\to Y$
of $k$-schemes possibly imposed with some condition for a more general
field or ring $k$. For example, the famous negative solution by Davis,
Matiyasevich, Putnam, and Robinson to Hilbert's tenth problem says
that the existence of a morphism $\Spec\ZZ\to Y$ with $Y$ an affine
scheme of finite type over $\ZZ$ is undecidable. One of the other
undecidability results in this direction is the one of Kanel'-Belov
and Chilikov \cite{kaneltextquoteright-belov2019onthe} (see also
\cite{kollar2020pellsurfaces}) that the existence of an embedding
$X\hookrightarrow Y$ of varieties over $\RR$ or $\overline{\QQ}$
is undecidable. 

The rest of the paper is organized as follows. In Section \ref{sec:Preliminaries},
we set up our basic convention. In particular, we clarify what we
mean by saying that some object (for example, a scheme or an invertible
sheaf) is given. In Section \ref{sec:Semidecidability Iso Prob},
we show that the isomorphism problem of projective schemes is semi-decidable.
In Section \ref{sec:Hilbert-schemes}, we explain how to compute the
Hilbert scheme for each polynomial. In Section \ref{sec:Projective-equaivalence},
we apply computation of the Hilbert scheme to show that it is decidable
whether two projective schemes embedded in the same projective space
is projectively equivalent. Although eash result in sections \ref{sec:Semidecidability Iso Prob}
to \ref{sec:Projective-equaivalence} would be known to specialists,
we include them for the sake of reader's convenience. The reader who
knows these materials well may skip these sections. In Section \ref{sec:Hom and Iso schemes},
we explain how to compute the Hom scheme and the Iso scheme for each
polynomial. In Section \ref{sec:One-dimensional-schemes}, we show
the decidability of the isomorphism problem for one-dimensional projective
schemes and the one for one-dimensional quasi-projective reduced schemes.
In Section \ref{sec: genetal type Fano}, we do the same for the case
of smooth irreducible varieties with a big canonical sheaf or a big
anti-canonical sheaf. In Section \ref{sec:Computing-intersection-numbers},
we explain how to compute intersection numbers on a smooth irreducible
projective variety. In Section \ref{sec:Positivity}, we discuss decidability
of positivity properties of invertible sheaves. In particular, we
show that ampleness of an invertible sheaf on a smooth variety is
decidable and that the nef cone and the pseudo-effective cone are
approximated by rational polyhedral cones with arbitrary precision.
In Section \ref{sec: K3}, we discuss the isomorphism problem for
K3 surfaces as well as smooth varieties with a rational polyhedral
nef cone.

\subsection*{Acknowledgments}

The author is grateful to Bjorn Poonen and Burt Totaro for pointing
out references and for valuable suggestions. He would also like to
thank Ichiro Shimada, Sho Ejiri, and an anonymous referee for helpful
comments.

\section{\label{sec:Preliminaries}Preliminaries}

Throughout the paper, we work over the field of complex algebraic
numbers, $\overline{\QQ}\subset\CC$, which is denoted by $k$. As
explained in \cite[Section 2.1]{simpson2008algebraic}, elements of
this field are expressed by finite data and four basic arithmetic
operations on them, addition, subtraction, multiplication and division,
are algorithmically computable. We can also algorithmically decide
whether or not two expressions give the same number. It follows that
we can also express polynomials with coefficients in $k$ by finite
data and algorithmically compute their addition, subtraction and multiplication.
We can also compute the Gröbner basis of an ideal in a polynomial
ring $k[x_{1},\dots,x_{m}]$. Thus we can also make various computation
based on the Gröbner basis. For example, we can algorithmically check
whether or not an ideal is contained in another ideal in the same
polynomial ring (this is an application of the ideal membership test;
see \cite[15.10.1]{eisenbud1995commutative}). We also note that the
elements of $k$ as well as the elements of $k[x_{1},\dots,x_{m}]$
are enumerable. 

When we say that a projective scheme $X$ is given, we mean that we
are given finitely many homogeneous polynomials $f_{1},\dots,f_{l}\in k[x_{1},\dots,x_{m}]$
such that $X$ is the closed subscheme of $\PP^{m-1}=\Proj k[x_{1},\dots,x_{m}]$
defined by the ideal $(f_{1},\dots,f_{l})$. In particular, we are
given an embedding $\iota\colon X\hookrightarrow\PP^{m-1}$ into a
projective space, the induced very ample invertible sheaf $\iota^{*}\cO_{\PP^{m-1}}(1)$
and the homogeneous coordinate ring $R_{X}:=k[x_{1},\dots,x_{m}]/(f_{1},\dots,f_{l})$,
which is also denoted by $R$ omitting the subscript $X$. From these
data, we can compute the standard affine charts $X_{i}:=X\cap\{x_{i}\ne0\}$
for $1\le i\le m$, which cover $X$. Indeed, if $k[x_{1},\dots,\check{x_{i}},\dots,x_{m}]$
is the polynomial ring with $x_{i}$ removed and if $f_{j}^{(i)}$
denotes the polynomial obtained from $f_{j}$ by substituting $1$
for $x_{i}$, then $X_{i}$ is the closed subscheme of $\AA^{m-1}=\Spec k[x_{1},\dots,\check{x_{i}},\cdots,x_{m}]$
defined by $f_{1}^{(i)},\dots,f_{l}^{(i)}$. 

For a projective scheme $X\subset\PP^{m-1}$ defined by $f_{1},\dots,f_{l}$,
we suppose that every coherent sheaf of $X$ (in particular, an invertible
sheaf) is represented by a finitely generated module over $R=R_{X}$.
In turn, every finitely generated $R$-module is represented by a
matrix $A\in\M_{r\times s}(R)$ which defines a free presentation
of $M$,
\[
\bigoplus_{j=1}^{s}R(b_{j})\xrightarrow{A}\bigoplus_{i=1}^{r}R(a_{i})\to M\to0.
\]
Here maps are supposed to be degree-preserving and $R(a)$ denotes
the graded free $R$-module of rank one defined by $R(a)_{c}=R_{a+c}$. 

When two projective schemes $X\subset\PP^{m-1}$ and $Y\subset\PP^{n-1}$
are given, we can embed the product $X\times Y$ into $\PP^{mn-1}$
via the Segre embedding $\PP^{m-1}\times\PP^{n-1}\hookrightarrow\PP^{mn-1}$.
When we say that a morphism $f\colon X\to Y$ is given, we mean that
its graph $\Gamma_{f}\subset X\times Y$ is given as a closed subscheme
of $\PP^{mn-1}$. 

\section{\label{sec:Semidecidability Iso Prob}Semi-decidability of the isomorphism
problem of projective schemes}

In this section, we show the probably well-known fact that the isomorphism
problem of projective schemes is semi-decidable; there exists an algorithm
such that, when two projective schemes are given as an input, then
the algorithm stops after finitely many steps if and only if these
schemes are isomorphic. The algorithm given in this section is a very
naive one and would be very inefficient. An approach via Iso schemes
would give a more efficient algorithm (see Remark \ref{rem:2nd proof semidecidability}).
\begin{rem}
Poonen pointed out to the author that the isomorphism problem of finite-type
$k$-schemes is also semi-decidable and it appears well-known. Roughly,
the proof is by checking whether the given schemes have the ``same''
affine open coverings. Our proof below for projective schemes is more
along our basic strategy in terms of graphs. Arguments in it will
be repeated in computation of Iso schemes in Section \ref{sec:Hom and Iso schemes}. 
\end{rem}

Let $X\subset\PP^{m-1}$ and $Y\subset\PP^{n-1}$ be projective schemes.
We first enumerate all the closed subschemes of $X\times Y$. To do
so, we enumerate all the finite sequences $f_{1},\dots,f_{l}$ of
homogeneous polynomials in $k[w_{ij}\mid1\le i\le m,\,1\le j\le n]$.
For each positive integer $i$, let $I_{i}$ be the ideal generated
by the $i$-th sequence and let $Z_{i}\subset\PP^{nm-1}$ be the closed
subscheme corresponding to $I_{i}$. Thus we obtain the sequence $Z_{i}$,
$i>0$ of closed subschemes such that for every closed subscheme $Z\subset\PP^{nm-1}$,
there exists $i>0$ such that $Z=Z_{i}$. For each $i$, we can check
whether or not $Z_{i}$ is included in $X\times Y$. Removing the
ones not included in $X\times Y$, we can algorithmically produce
every closed subscheme of $X\times Y$ one by one. If we prefer, we
may remove redundancies to get a sequence where every closed subscheme
of $X\times Y$ appears exactly once. We let $Z_{i}$, $i>0$ be thus
obtained sequence of closed subschemes of $X\times Y$. 
\begin{prop}
\label{prop: iso semidecidable}The isomorphism problem of projective
schemes is semi-decidable
\end{prop}

\begin{proof}
For each integer $i>0$, from Lemma \ref{lem:iso decidable} below,
we can algorithmically check whether $Z_{i}$ is the graph of an isomorphism
$X\xrightarrow{\sim}Y$. As soon as one finds that this is the case,
we stop this algorithm.
\end{proof}
\begin{lem}
\label{lem:iso decidable}We can algorithmically check whether or
not a given closed subscheme $\Gamma\subset X\times Y$ is the graph
of an isomorphism $X\xrightarrow{\sim}Y$. 
\end{lem}

\begin{proof}
We need to check whether the two projections $\Gamma\to X$ and $\Gamma\to Y$
are both isomorphisms. We discuss only the former projection, denoting
it by $f$. Let $X=\bigcup X_{j}$ be the standard affine open covering
and let $R_{j}$ be the coordinate ring of $X_{j}$. Let $\Gamma_{j}$
be the preimage of $X_{j}$ by the morphism $\Gamma\to X$, which
is a closed subscheme of $\PP_{R_{j}}^{n-1}=\Proj R_{j}[y_{1},\dots,y_{n}]$.
The morphism $f\colon\Gamma\to X$ is an isomorphism if and only if
every 
\[
f_{j}:=f|_{\Gamma_{j}}\colon\Gamma_{j}\to X_{j}
\]
is an isomorphism. We can compute the coherent sheaf $\Omega_{\Gamma_{j}/X_{j}}$
of differentials (see Remark \ref{rem:cotangent}) and check whether
or not it is the zero sheaf. Thus we can algorithmically check whether
or not $f_{j}$ is unramified. If it is not unramified, then it is
not an isomorphism. Suppose that $f_{j}$ is unramified. Then it is
also a finite morphism. We compute an $R_{j}$-module $M_{j}$ such
that $\widetilde{M_{j}}:=(f_{j})_{*}\cO_{\Gamma_{j}}$ (see Remark
\ref{rem:push}). Consider the following three conditions:
\begin{enumerate}
\item $\Supp(M_{j})=X_{j}$.
\item $V(\Fitt_{1}(M_{j}))=\emptyset$, where $\Fitt_{1}(M_{j})$ denote
the first Fitting ideal of $M_{j}$.
\item $\mathrm{Tor}^{R_{j}}((R_{j})_{\red},M_{j}\otimes(R_{j})_{\red})=0$.
\end{enumerate}
The second condition means that for every point $x\in X_{j}$, the
stalk $(\widetilde{M_{j}})_{x}$ is generated by one element as an
$\cO_{X_{j},x}$-module. From \cite[II, Exercise 5.8]{hartshorne1977algebraic},
the first two conditions together are equivalent to that $M_{j}\otimes(R_{j})_{\red}$
is a flat $(R_{j})_{\red}$-module of constant rank one. Under these
conditions, the third condition means that $M_{j}$ is a flat $R_{j}$-module
(see \cite[Th.\ 22.3]{matsumura1989commutative} or \cite[tag 051C]{thestacksprojectauthors2021thestacks}).
We conclude that these three conditions all hold if and only if the
finite unramified morphism $f_{j}\colon\Gamma_{j}\to X_{j}$ is surjective
and flat of constant rank one, that is, an isomorphism. 
\end{proof}
\begin{rem}[{cf.~\cite[Prop. 5.7]{stillmancomputing}}]
\label{rem:cotangent}Let $R$ be a commutative ring and let 
\[
X=\Proj R[x_{1},\dots,x_{r}]/(f_{1},\dots,f_{l})
\]
be a projective scheme over $R$, where $f_{1},\dots,f_{l}$ are homogeneous
polynomials. Then the cotangent sheaf $\Omega_{X/R}$ is associated
to the homology module of the sequence
\[
\bigoplus_{i=1}^{l}S(-\deg f_{i})\xrightarrow{(\partial f_{i}/\partial x_{j})_{i,j}}S(-1)^{\oplus r}\xrightarrow{(x_{1}\cdots x_{r})}S.
\]
Stillman's notes cited above treat the case where $R$ is a field,
but this lemma holds for an arbitrary $R$. This is a straightforward
consequence of \cite[Prop.\ 8.12 and Th.\ 8.13]{hartshorne1977algebraic}. 
\end{rem}

\begin{rem}
\label{rem:push}Let $X\subset\PP_{R}^{n}$ be a projective scheme
over a commutative ring $R$, let $\pi\colon X\to\Spec R$ be the
structure morphism and let $\cM$ be a coherent sheaf on $X$. Algorithms
computing the pushfoward $\pi_{*}\cM$ (and more generally, higher
direct images $R^{i}\pi_{*}\cM$) are explained in \cite{smith1998computing,eisenbud2008relative}. 
\end{rem}

\section{Hilbert schemes\label{sec:Hilbert-schemes}}

In this section, we discuss how to compute Hilbert schemes of general
projective schemes and their universal families. Their computability,
Proposition \ref{prop:univ fam}, has been already proved in \cite[Lem.\ 8.23]{poonen2015computing}.
As this result is the core of our approach, we explain it in more
details below. 

\subsection{The Hilbert scheme of a projective space\label{subsec: Hilb P}}

Bayer \cite{bayer1982thedivision} explained how to compute equations
defining the Hilbert scheme $\Hilb_{P}(\PP^{r-1})$ for each Hilbert
polynomial $P$ as a closed subset of a Grassmaniann variety. It turned
out that his equations also give the right scheme structure of the
Hilbert scheme. We recall this description of the Hilbert scheme $\Hilb_{P}(\PP^{r-1})$,
closely following the presentation by Iarrobino and Kleiman in \cite[Appendix C]{iarrobino1999sumsof}
but with emphasis on algorithmic aspects.

Throughout this section, we fix a positive integer $r>0$. Let $R:=k[x_{1},\dots,x_{r}]=\bigoplus_{i\ge0}R_{i}$
with $R_{i}$ denoting the degree-$i$ part and let $\PP^{r-1}=\Proj R$,
the $(r-1)$-dimensional projective space. For a closed subscheme
$Z\subset\PP^{r-1}$ defined by a homogeneous ideal $I\subset R$,
the \emph{Hilbert polynomial} of $Z$ is a polynomial $P\in\QQ[t]$
such that $P(i)=\dim_{k}(R/I)_{i}$ for $i\gg0$, where $(R/I)_{i}$
denotes the degree $i$ part of the graded ring $R/I$. A polynomial
is said to be a \emph{Hilbert polynomial} if it is the Hilbert polynomial
of some closed subscheme $Z\subset\PP^{r-1}$. For each Hilbert polynomial
$P$, the \emph{Hilbert scheme} $\Hilb_{P}(\PP^{r-1})$ for $\P$
is the moduli scheme of closed subschemes $Z\subset\PP^{r-1}$ with
the Hilbert polynomial $P$. 

For a Hilbert polynomial $P$, there exists a unique sequence of positive
integers, $0<a_{0}\le\cdots\le a_{k}$ such that $0\le k\le r-2$
and 
\[
\binom{r+t-1}{r-1}-P(t)=\binom{t-a_{0}+r-1}{r-1}+\cdots+\binom{t-a_{k}+r-1-k}{r-1-k}.
\]
We can algorithmically compute these integers $a_{0},\dots,a_{k}$
from the polynomial $P$. The \emph{Gotzmann number }$\varphi(P)$
of $P$ is defined to be $a_{k}$. 

We now fix a Hilbert polynomial $P$ and an integer $d\ge\varphi(P)$.
Let $r_{d}:=\dim_{k}R_{d}$, $p:=P(d)$ and $p^{\vee}:=r_{d}-p$ .
Let $\Grass_{p^{\vee}}(R_{d})$ be the Grassmannian parameterizing
$p^{\vee}$-dimensional subspaces of the $r_{d}$-dimensional vector
space $R_{d}$. There exists a closed embedding
\begin{align*}
\Hilb_{P}(\PP^{r-1}) & \hookrightarrow\Grass_{p^{\vee}}(R_{d}),\\{}
[Z] & \mapsto[(I_{Z})_{d}]
\end{align*}
where $I_{Z}\subset R$ is the saturated ideal of $Z$ and $(I_{Z})_{d}$
is its degree-$d$ part. In particular, the closed subscheme $Z$
is recovered from the subspace $(I_{Z})_{d}\subset R_{d}$. Indeed
$Z$ is defined by the ideal $((I_{Z})_{d})\subset R$ generated by
$(I_{Z})_{d}$. 

Let $M_{d}$ be the set of the monomials of degree $d$, which is
a basis of $R_{d}$. The Grassmannian $\Grass_{p^{\vee}}(R_{d})$
has the standard affine open covering 
\[
\Grass_{p^{\vee}}(R_{d})=\bigcup_{\substack{K\subset M_{d}\\
\sharp K=p
}
}U_{K}.
\]
Each affine chart $U_{K}$ is isomorphic to the $p\cdot p^{\vee}$-dimensional
affine space $\AA^{p\cdot p^{\vee}}$. In what follows, we identify
$M_{d}$ with $\{1,2,\dots,r_{d}\}$ say by the lex order. Then, the
affine chart $U_{K}=\AA^{p\cdot p^{\vee}}$ is the space of $r_{d}$-by-$p^{\vee}$
matrices 
\begin{equation}
A=(a_{i,j})_{i\in M_{d},1\le j\le p^{\vee}}\label{eq:mat A}
\end{equation}
such that the $p^{\vee}$-by-$p^{\vee}$ submatrix 
\[
(a_{i,j})_{i\in M_{d}\setminus K,1\le j\le p^{\vee}}
\]
is the identity matrix. Note that the $j$-th column of $A$ corresponds
to the homogeneous polynomial $\sum_{i=1}^{r_{d}}a_{ij}m_{i}$, where
$m_{i}$ denotes the $i$-th monomial in $M_{d}$. Thus we can write
the coordinate ring of $U_{K}$ as 
\[
k[U_{K}]=k[u_{i,j}\mid i\in K,1\le j\le p^{\vee}],
\]
where $u_{i,j}$ are indeterminates corresponding to entries $a_{i,j}$
above, respectively. For example, if $K$ consists of the last $p$
monomials in $M_{d}$, then a matrix $A$ as above is of the form:
\[
A=\begin{pmatrix}1 &  & 0\\
 & \ddots\\
0 &  & 1\\
a_{p^{\vee}+1,1} & \cdots & a_{p^{\vee}+1,p^{\vee}}\\
\vdots & \ddots & \vdots\\
a_{r_{d},1} & \cdots & a_{r_{d},p^{\vee}}
\end{pmatrix}
\]
The $p\cdot p^{\vee}$ free entries in the last $p$ rows serve as
coordinates of the affine space $U_{K}$. For general $K$, if we
write $A=(\ba_{1}\cdots\ba_{p^{\vee}})$ with column vectors $\ba_{i}$,
then the matrix $A$ corresponds to the subspace $\langle\ba_{1},\dots,\ba_{p^{\vee}}\rangle\subset R_{d}$.
A matrix as above is also regarded as a linear map 
\[
k^{p^{\vee}}\to k^{r_{d}}=R_{d}.
\]
For each $i$ with $1\le i\le r$, let $B_{i}$ be the $r_{d+1}$-by-$p^{\vee}$
matrix corresponding to the composite map 
\[
k^{p^{\vee}}\xrightarrow{A}R_{d}\xrightarrow{\times x_{i}}R_{d+1}=k^{r_{d+1}}.
\]
These matrices are easily computed from $A$. Indeed their nonzero
entries are the ones of $A$ suitably arranged. Finally we define
the $r_{d+1}$-by-$rp^{\vee}$ matrix 
\[
B:=(B_{1}|\cdots|B_{r})
\]
by lining $B_{i}$'s horizontally. If $V_{A}\subset R_{d}$ is the
subspace corresponding to $A$, then the image of $B$ regarded as
the map $k^{rp^{\vee}}\to R_{d+1}$ is nothing but $R_{1}\cdot V_{A}\subset R_{d+1}$.
We have that the point $[V_{A}]\in\Grass_{p^{\vee}}(R_{d})$ lies
in $\Hilb_{P}(\PP^{r-1})$ if and only if the inequality 
\[
\dim_{k}R_{1}\cdot V_{A}\le r_{d+1}-P(d+1)=:q^{\vee}
\]
holds. Note that the inequality is equivalent to the equality, since
the opposite inequality $\ge$ always holds. Thus, on the affine chart
$U_{K}\subset\Grass_{p^{\vee}}(R_{d})$, the Hilbert scheme $\Hilb_{P}(\PP^{r-1})$
is cut out, at least set-theoretically, by the $(q^{\vee}+1)$-by-$(q^{\vee}+1)$
minors of the matrix $B$; each such minor is a polynomial in coordinates
$a_{i,j}$ of $U_{K}=\AA^{p\cdot p^{\vee}}$. It turns out that this
is the case also scheme-theoretically. Precisely:
\begin{prop}
\label{prop:Hilb eqs}The closed subscheme $\Hilb_{P}(\PP^{r-1})\cap U_{K}$
of $U_{K}=\AA^{p\cdot p^{\vee}}$ is defined by the $(q^{\vee}+1)$-by-$(q^{\vee}+1)$
minors of the matrix $B$. 
\end{prop}

\subsection{The universal family}

We can also compute the universal family 
\[
\cZ_{K}\subset(\Hilb_{P}(\PP^{r-1})\cap U_{K})\times\PP^{r-1}
\]
over $\Hilb_{P}(\PP^{r-1})\cap U_{K}$ as follows. Recall that $U_{K}=\AA^{p\cdot p^{\vee}}$
has coordinates $u_{i,j}$ ($i\in K$, $1\le j\le p^{\vee}$) and
consider the universal $r_{d}$-by-$p^{\vee}$ matrix 
\begin{equation}
\cA=\cA_{K}:=(u_{i,j})_{i\in M_{d},1\le j\le p^{\vee}}.\label{eq:univ mat}
\end{equation}
Here, for $i\in K$, the entry $u_{i,j}$ is the indeterminate $u_{i,j}\in k[U_{K}]$
and, for $i\in M_{d}\setminus K$, $u_{i,j}$ is defined to be either
1 or 0 so that the $p^{\vee}$-by-$p^{\vee}$ matrix 
\[
(u_{i,j})_{i\in M_{d}\setminus K,1\le j\le p^{\vee}}
\]
is the identity matrix. At each point $(a_{i,j})\in U_{K}$, the universal
matrix to the matrix $A$ in (\ref{eq:mat A}). The $j$-th column
of $\cA$ defines the universal $j$-th polynomial
\[
h_{K,j}:=\sum_{i=1}^{r_{d}}u{}_{i,j}m_{i}=n_{i}+\sum_{i\in K}u_{i,j}m_{i}\in k[U_{K}][x_{1},\dots,x_{r}],
\]
where $m_{i}$ is the $i$-th monimial in $M_{d}$ as before and $n_{i}$
is the $i$-th monomial in $M_{d}\setminus K$. They are homogeneous
of degree $d$ with respect to variables $x_{1},\dots,x_{r}$. 
\begin{prop}
Let $I_{K}\subset k[U_{K}]$ be the defining ideal of $\Hilb_{P}(\PP^{r-1})\cap U_{K}\subset U_{K}$,
which can be computed as in Proposition \ref{prop:Hilb eqs}. Then
the universal family $\cZ_{K}$ is written as
\[
\cZ_{K}=\Proj\frac{(k[U_{K}]/I_{K})[x_{1},\dots,x_{r}]}{(h_{K,1},\dots,h_{K,p^{\vee}})}.
\]
\end{prop}

\subsection{The Hilbert scheme of a general projective scheme}

Next we consider the Hilbert scheme $\Hilb_{P}(X)$ of a projective
scheme $X\subset\PP_{k}^{r-1}$. This is the moduli scheme of those
closed subschemes $Z\subset X$ that have the Hilbert polynomial $P$
as a closed subscheme of $\PP_{k}^{r-1}$. Let $Q$ be the Hilbert
polynomial of $X$. If $d\ge\varphi(Q)$, then $I_{X}$ is $d$-regular.
In particular, the truncated ideal $(I_{X})_{\ge d}$, which is the
part of the saturated ideal $I_{X}$ with degree $\ge d$, is generated
by the degree-$d$ part $(I_{X})_{d}$. Now we choose the integer
$d$ to be $\max\{\varphi(P),\varphi(Q)\}$. For a closed subscheme
$Z\subset\PP_{k}^{r-1}$ with the Hilbert polynomial $P$, we have
$Z\subset X$ if and only if $(I_{Z})_{d}\supset(I_{X})_{d}$. Let
\[
f_{1}=\begin{pmatrix}f_{1,1}\\
\vdots\\
f_{r_{d},1}
\end{pmatrix},\dots,f_{l}=\begin{pmatrix}f_{1,l}\\
\vdots\\
f_{r_{d},l}
\end{pmatrix}\in k^{r_{d}}=R_{d}
\]
be a basis of $(I_{X})_{d}$. Note that we can explicitly construct
such a basis from the given finitely many defining polynomials of
$X$. Firstly, there is an algorithm to compute the saturation $(I:(x_{1},\dots,x_{r})^{\infty})$
of an homogeneous ideal $I\subset k[x_{1},\dots,x_{r}]$, see \cite[page 360]{eisenbud1995commutative}.
If $g_{1},\dots,g_{m}$ are thus computed generators of the saturated
ideal $I_{X}$, which we may assume have degree $\le d$ (since those
of degree $>d$ are redundant), then $(I_{X})_{d}$ is generated by
elements of the form $xg_{i}$, where $x$ is a monomial of degree
$d-\deg g_{i}$. We can choose a basis from them in a standard linear
algebra procedure. As before, a point of $U_{K}=\AA^{p\cdot p^{\vee}}$
is identified with a $r_{d}$-by-$p^{\vee}$ matrix $A$. To such
a matrix, we associate the $r_{d}$-by-$(p^{\vee}+l)$ matrix $C_{A}=(A|f_{1}|\cdots|f_{l})$.
For example, if $K$ is the first $p$ monomials in $M_{d}$, then
\[
C_{A}=\begin{pmatrix}1 &  & 0 & f_{1,1} & \cdots & f_{1,l}\\
 & \ddots\\
0 &  & 1 & \vdots & \ddots & \vdots\\
a_{p^{\vee}+1,1} & \cdots & a_{p^{\vee}+1,p^{\vee}} & \vdots &  & \vdots\\
\vdots & \ddots & \vdots\\
a_{r_{d},1} & \cdots & a_{r_{d},p^{\vee}} & f_{r_{d},1} & \cdots & f_{r_{d},l}
\end{pmatrix}.
\]
For a closed subscheme $Z\subset\PP^{r-1}$ with $[Z]\in U_{K}$,
the following conditions are equivalent:
\begin{enumerate}
\item $(I_{Z})_{d}\supset(I_{X})_{d}$. 
\item $\rank C_{A}=p^{\vee}$. 
\item all the $(p^{\vee}+1)$-by-$(p^{\vee}+1)$ minors of $C_{A}$ vanish.
\end{enumerate}
The minors in the last condition are polynomials in the coordinates
of $U_{K}$. 
\begin{prop}
The closed subscheme 
\begin{equation}
\Hilb_{P}(X)_{K}:=\Hilb_{P}(X)\cap U_{K}\label{eq:Hilb_K}
\end{equation}
of $U_{K}=\AA^{p\cdot p^{\vee}}$ is defined by the defining polynomials
of $\Hilb_{P}(\PP^{r-1})$ given in Proposition \ref{prop:Hilb eqs}
and the above minors of $C_{A}$.
\end{prop}

We can compute the universal family $\cZ_{X,K}\subset\Hilb_{P}(X)_{K}\times X$
over $\Hilb_{P}(X)_{K}$ in a similar way as in the case of $\cZ_{K}$.

If $I_{X,K}\subset k[U_{K}]$ denotes the defining ideal of $\Hilb_{P}(X)\cap U_{K}$,
then 
\begin{equation}
\begin{aligned}\cZ_{X,K} & =\Proj\frac{(k[U_{K}]/I_{X,K})[x_{1},\dots,x_{r}]}{(f_{1},\dots,f_{l},h_{K,1},\dots,h_{K,p^{\vee}})}\\
 & \subset\Proj\frac{(k[U_{K}]/I_{X,K})[x_{1},\dots,x_{r}]}{(f_{1},\dots,f_{l})}=\Hilb_{P}(X)_{K}\times X.
\end{aligned}
\label{eq:Z X K}
\end{equation}
As a conclusion of the above computation, we have:
\begin{prop}[{\cite[Lem.\ 8.23]{poonen2015computing}}]
\label{prop:univ fam}There is an algorithm such that given a projective
variety $X\subset\PP^{r-1}$ and a Hilbert polynomial $P$, then it
outputs a positive integer $d$ and computes the closed subscheme
$\Hilb_{P}(X)\subset\Grass_{p^{\vee}}(R_{d})$ in terms of defining
equations on each affine chart $U_{K}$ together with defining equations
of the universal families $\cZ_{X,K}\subset\Hilb_{P}(X)_{K}\times X$. 
\end{prop}

\subsection{The action of $\protect\GL_{r}$\label{subsec: GL_r action}}

Since it will be used to prove the decidability of projective equivalence
in Section \ref{sec:Projective-equaivalence}, we describe the natural
action of the general linear group $\GL_{r}$ on the Hilbert scheme
$\Hilb_{P}(\PP^{r-1})$. The action of $\GL_{r}$ on $\PP^{r-1}$
induces an action of $\GL_{r}$ on $\Hilb_{P}(\PP^{r-1})$. The Hilbert
scheme has the open covering $\Hilb_{P}(\PP^{r-1})=\bigcup_{K}U_{K}$
and we have an explicit presentation of $U_{K}$ for each $K$. 
\begin{defn}
We define $V_{K,K'}$ to be the preimage of $U_{K'}$ by the morphism
\[
\mu\colon\GL_{r}\times U_{K}\to\Hilb_{P}(\PP^{r-1}).
\]
\end{defn}

We explain how to compute an explicit presentation of the affine scheme
$V_{K,K'}$ as well as the morphism 
\begin{equation}
\mu_{K,K'}:=\mu|_{V_{K,K'}}\colon V_{K,K'}\to U_{K'}.\label{eq:mu K K'}
\end{equation}

Let $S_{r}$ be the coordinate ring of $\GL_{r}$, that is, the localization
$k[\underline{s}]_{D}$ of the polynomial ring $k[\underline{s}]=k[s_{i,j}\mid1\le i,j\le r]$
by the determinant $D=\det(s_{i,j})$. We have the universal matrix
\[
\begin{pmatrix}s_{1,1} & \cdots & s_{1,r}\\
\vdots & \ddots & \vdots\\
s_{r,1} & \cdots & s_{r,r}
\end{pmatrix}\in\GL_{r}(S_{r}).
\]
The action $\GL_{r}\times\AA^{r}\to\AA^{r}$ is given by the following
ring map: 
\begin{align*}
k[x_{1},\dots,x_{r}] & \to S_{r}[x_{1},\dots,x_{r}]\\
x_{i} & \mapsto\sum_{j=1}^{r}s_{i,j}x_{j}
\end{align*}

As before, we fix $d\ge\varphi(P)$. We write a monomial in $k[x_{1},\dots,x_{r}]$
as $x^{e}=x_{1}^{e_{1}}\cdots x_{r}^{e_{r}}$ with multi-index notation.
The last map sends a monomial $x^{e}$ of degree $d$ to the polynomial,
\[
\prod_{i=1}^{d}\left(\sum_{j=1}^{r}s_{i,j}x_{j}\right)^{e_{i}}=:\sum_{e'}\eta_{e,e'}x^{e'},
\]
which is homogeneous of degree $d$ both in $x_{1},\dots,x_{r}$ and
in $s_{i,j}$. Here $\eta_{e',e}$ is a homogenous polynomial of degree
$d$ in $k[\underline{s}]$. For each $e$, we can compute $\eta_{e,e'}$'s
explicitly. We get the map
\begin{align*}
\GL_{r} & \to\GL_{r_{d}}\\
(a_{i,j})_{1\le i,j\le r} & \mapsto(\eta_{e,e'}(\underline{a}))_{e,e'\in M_{d}},
\end{align*}
which induces an action of $\GL_{r}$ on $\AA^{r_{d}}$. In terms
of coordinate rings, this is given by 
\[
S_{r_{d}}\to S_{r},\,s_{e,e'}\mapsto\eta_{e,e'}(s_{1,1},\dots,s_{r,r}).
\]

Let $K,K'\subset M_{d}$ be two subsets with $\sharp K=\sharp K'=p$.
Let $\cA_{K}=(u_{i,j})$ be the universal $r_{d}$-by-$p^{\vee}$
matrix for $K$ (see (\ref{eq:univ mat})). Compute the matrix product
\[
\cB:=(\eta_{e,e'})\cdot\cA_{K}\in\M_{r_{d}\times p^{\vee}}(k(\underline{s})[\underline{u}])
\]
with $k(\underline{s})$ the fraction field of $k[\underline{s}]$
and $k(\underline{s})[\underline{u}]=k(\underline{s})[u_{i,j}\mid i\in K,1\le j\le p^{\vee}]$.
By Gaussian elimination in the field $k(\underline{s},\underline{u})$,
we can algorithmically and uniquely transform $\cB$ to a matrix belonging
to $U_{K'}(k(\underline{s},\underline{u}))$ by applying elementary
column operations finitely many times; let $\theta=(\theta_{i,j})\in\M_{r_{d}\times p^{\vee}}(k(\underline{s},\underline{u}))$
be the resulting matrix. Here the $p^{\vee}$-by-$p^{\vee}$ matrix
$(\theta_{i,j})_{i\in M_{d}\setminus K',1\le j\le r_{d}}$ is the
identity matrix. We get $p\cdot p^{\vee}$ rational functions $\theta_{i,j}\in k(\underline{s},\underline{u})$
for $i\in K'$ and $1\le j\le p^{\vee}$. We see that the rational
map $\GL_{r}\times U_{K}\dashrightarrow U_{K'}$ is given by:
\begin{align*}
k[\underline{u}] & \to k(\underline{s},\underline{u})\\
u_{i,j} & \mapsto\theta_{i,j}
\end{align*}
Let $V_{K,K'}\subset\GL_{r}\times U_{K}$ be the domain of this rational
map, that is, the preimage of $U_{K'}$ by the map $\GL_{r}\times U_{K}\to\Hilb_{P}(\PP^{r-1}).$
This is an affine scheme, since $\GL_{r}\times U_{K}$ and $U_{K'}$
are affine and $\Hilb_{P}(\PP^{r-1})$ is separated. The coordinate
ring $T_{K,K'}$ of $V_{K,K'}$ is obtained by adjoining $\theta_{i,j}$'s
to $S_{r}[\underline{u}]$. 

\section{\label{sec:Projective-equaivalence}Projective equivalence}

Two projective schemes $X$ and $Y$ embedded in the same projective
space $\PP^{r-1}$ are said to be \emph{projectively equivalent }if
there exists an invertible matrix $g\in\GL_{r}(k)$ such that $g(X)=Y$. 
\begin{prop}
\label{prop: embedded iso}Let $X$ and $Y$ be projective schemes
embedded in the same projective space $\PP^{r-1}$. Then we can algorithmically
check whether $X$ and $Y$ are projectively equivalent.
\end{prop}

\begin{proof}
We follow the notation of Section \ref{sec:Hilbert-schemes}. We can
compute the Hilbert polynomials of $X$ and $Y$ (see \cite[Sections 15.1.1 and 15.10.2]{eisenbud1995commutative})
and check whether they are the same. If they are different, then there
is no $g\in\GL_{r}(k)$ as in the proposition. Suppose that they are
the same and denote it by $P$. Let $d:=\varphi(P)$, the Gotzmann
number of $P$, let $I_{X}$ and $I_{Y}$ be the saturated ideals
of $X$ and $Y$ and let $(I_{X})_{d}$ and $(I_{Y})_{d}$ be their
degree-$d$ parts. We compute bases of $(I_{X})_{d}$ and $(I_{Y})_{d}$,
which are represented by $r_{d}$-by-$p^{\vee}$ matrices $A_{X}$
and $A_{Y}$ respectively. We then compute their reduced column echelon
forms and denote them by $B_{X}$ and $B_{Y}$. Let $K$ and $K'$
be the set of indices such that the corresponding rows of $B_{X}$
and $B_{Y}$ have pivots. Then $B_{X}$ and $B_{Y}$ define the points
$[X]\in U_{K}$ and $[Y]\in U_{K'}$ respectively. Consider the morphism
$\mu_{K,K'}\colon V_{K,K'}\to U_{K'}$ (see (\ref{eq:mu K K'})) and
the morphism 
\[
p\colon V_{K,K'}\hookrightarrow\GL_{r}\times U_{K}\xrightarrow{\text{projection}}U_{K}
\]
of affine schemes. We compute the closed subset
\[
(\mu_{K,K'})^{-1}([Y])\cap p^{-1}([X])\subset V_{K,K'},
\]
which is the set of pairs $(g,[X])$ such that $g\in\GL_{r}(k)$ and
$g(X)=Y$. Thus, there exists $g\in\GL_{r}(k)$ as in the proposition
if and only if this closed subset is not empty. 
\end{proof}

\section{Hom schemes and Iso schemes\label{sec:Hom and Iso schemes}}

For projective schemes $X$ and $Y$, the \emph{Hom scheme}, denoted
by $\ulHom(X,Y)$, and \emph{Iso scheme}, denoted by $\ulIso(X,Y)$,
are the moduli schemes of morphisms $X\to Y$ and isomorphisms $X\to Y$
respectively. From \cite[p.\ 16]{kollar1996rational}, we have an
open immersion 
\begin{align*}
\ulHom(X,Y) & \to\Hilb(X\times Y),\\{}
[f\colon X\to Y] & \mapsto[\Gamma_{f}]
\end{align*}
where $\Gamma_{f}\subset X\times Y$ is the graph of $f$. Namely,
we can identify $\ulHom(X,Y)$ with the locus of points $[Z]\in\Hilb(X\times Y)$
such that the first projection $Z\to X$ is an isomorphism. Similarly
we can identify $\ulIso(X,Y)$ with the locus of points $[Z]$ where
both the projections $Z\to X$ and $Z\to Y$ are isomorphisms. Thus,
if we embed also $\ulHom(Y,X)$ into $\Hilb(X\times Y)$ via the obvious
isomorphism $\Hilb(X\times Y)\cong\Hilb(Y\times X)$, then we have
\[
\ulIso(X,Y)=\ulHom(X,Y)\cap\ulHom(Y,X).
\]

We now fix embeddings $X\subset\PP^{m-1}$ and $Y\subset\PP^{n-1}$
and embed the product $X\times Y$ into $\PP^{mn-1}$ by the Segre
embedding. Then the Hilbert scheme $\Hilb(X\times Y)$ decomposes
into the disjoint union of countably many open and closed subschemes
as 
\[
\Hilb(X\times Y)=\coprod_{P}\Hilb_{P}(X\times Y).
\]
Here $P$ runs over Hilbert polynomials. Recall that $\Hilb_{P}(X\times Y)_{K}$
denotes $\Hilb_{P}(X\times Y)\cap U_{K}$, see (\ref{eq:Hilb_K}). 
\begin{defn}
We define 
\begin{align*}
\ulHom_{P}(X,Y) & :=\ulHom(X,Y)\cap\Hilb_{P}(X\times Y),\\
\ulHom_{P}(X,Y)_{K} & :=\ulHom(X,Y)\cap\Hilb_{P}(X\times Y)_{K},\\
\ulIso_{P}(X,Y) & :=\ulIso(X,Y)\cap\Hilb_{P}(X\times Y),\\
\ulIso_{P}(X,Y)_{K} & :=\ulIso(X,Y)\cap\Hilb_{P}(X\times Y)_{K}.
\end{align*}
\end{defn}

We have the open coverings,
\begin{align*}
\ulHom_{P}(X,Y) & =\bigcup_{K}\ulHom_{P}(X,Y)_{K},\\
\ulIso_{P}(X,Y) & =\bigcup_{K}\ulIso_{P}(X,Y)_{K}.
\end{align*}

\begin{defn}
\label{def: Hilb poly of an iso}We define the \emph{Hilbert polynomial
}of an isomorphism $f\colon X\to Y$ to be the Hilbert polynomial
of its graph $\Gamma_{f}\subset X\times Y$ as a closed subscheme
of $\PP^{mn-1}$. Namely the Hilbert polynomial of $f$ is the polynomial
$P$ such that $[f]\in\ulIso_{P}(X,Y)$. 
\end{defn}

When we show the decidability of the isomorphism problem for several
classes of projective schemes, our strategy will be to construct finitely
many polynomials $P_{1},\dots,P_{l}$ from given projective schemes
$X$ and $Y$ that satisfy the following property; if $X$ and $Y$
are isomorphic, then there exists an isomorphism $f\colon X\to Y$
having the Hilbert polynomial $P_{i}$ for some $i\in\{1,\dots,l\}$.
Namely $X$ and $Y$ are isomorphic if and only if $\bigcup_{i=1}^{l}\ulIso_{P_{i}}(X,Y)\ne\emptyset$.
Then what remains to do is to compute $\ulIso_{P_{i}}(X,Y)$ for every
$i\in\{1,\dots,l\}$.

For each $P$ and $K$, we have 
\[
\ulIso_{P}(X,Y)_{K}=\ulHom_{P}(X,Y)_{K}\cap\ulHom_{P}(Y,X)_{K}.
\]
Thus, computation of $\ulIso_{P}(X,Y)_{K}$ is reduced to the one
of $\ulHom_{P}(X,Y)_{K}$. We focus on the latter computation in what
follows. Replacing $X$ in (\ref{eq:Z X K}) with $X\times Y$, we
get explicit presentation of the universal family 
\[
\cZ_{X\times Y,K}\subset\Hilb_{P}(X\times Y)_{K}\times X\times Y.
\]
From discussion above, the open subscheme $\ulHom_{P}(X,Y)_{K}$ of
$\Hilb_{P}(X\times Y)_{K}$ is the largest open subscheme $U$ such
that the composite morphism
\[
g\colon\cZ_{X\times Y,K}\hookrightarrow\Hilb_{P}(X\times Y)_{K}\times X\times Y\xrightarrow{\text{projection}}\Hilb_{P}(X\times Y)_{K}\times X
\]
is an isomorphism over $U\times X$. Since we are given the embedding
$X\subset\PP^{m-1}$, we have the standard affine open covering $X=\bigcup_{i=1}^{m}X_{i}$.
Let 
\[
g_{i}\colon g^{-1}(X_{i})\to\Hilb_{P}(X\times Y)_{K}\times X_{i}
\]
be restriction of $g$. This is a projective morphism with the target
being affine. Let $U_{i}\subset\Hilb_{P}(X\times Y)_{K}$ be the largest
open subset such that $g_{i}$ is an isomorphism over $U_{i}\times X_{i}$.
Then 
\[
U=\ulHom_{P}(X,Y)_{K}=\bigcap_{i=1}^{m}U_{i}.
\]
There exists an algorithm to compute each $U_{i}$:
\begin{prop}
Let $H$ and $X$ be affine schemes and let $f\colon Z\to H\times X$
be a projective morphism. Then there exists an algorithm to compute
the largest open subset $U\subset H$ such that $f$ is an isomorphism
over $U\times X$.
\end{prop}

\begin{proof}
We first compute the closed subset $C_{1}:=\Supp(f_{*}\Omega_{Z/H\times X})$
of $H\times X$ (see Remarks \ref{rem:cotangent} and \ref{rem:push}).
Next we compute the closed subset $C_{2}:=(H\times X)\setminus V$,
where $V$ is the invertible locus of $f_{*}\cO_{Z}$ (see Lemma \ref{lem: invertible locus}).
The desired open subset $U\subset H$ is $H\setminus p_{H}(C_{1}\cup C_{2})$,
where $p_{H}$ denotes the projection $H\times X\to H$. Indeed, obviously
$H\setminus p_{H}(C_{1}\cup C_{2})$ is contained in the desired subset.
On the other hand, $f$ is unramified (in particular, finite and affine)
over $H\setminus p_{H}(C_{1})$. That $f_{*}\cO_{Z}$ is invertible
on $H\setminus p_{H}(C_{1}\cup C_{2})$ means that $f$ is an isomorphism
over this open subset. 
\end{proof}
\begin{lem}
\label{lem: invertible locus}There is an algorithm to compute the
invertible locus of a coherent sheaf on an affine scheme, that is,
the largest open subset on which the sheaf is invertible. 
\end{lem}

\begin{proof}
Let $X$ be an affine scheme and let $\cM$ be a coherent sheaf on
$X$. We first compute the closed subset 
\[
C_{1}:=\overline{X\setminus\Supp(\cM)}\subset X
\]
(see Remark \ref{rem:closure}). Its complement $X\setminus C_{1}$
is the largest open subset of $X$ that is included in $\Supp(\cM)$.
We then put $\cM':=\cM|_{X_{\red}}$ and compute the closed subset
\[
C_{2}:=V(\Fitt_{1}(\cM')),
\]
where $\Fitt_{1}$ denotes the first Fitting ideal. Its complement
$X\setminus C_{2}$ is the locus of points $x\in X$ where $\cM'_{x}$
is generated by one element as an $\cO_{X,x}$-module. Finally we
compute 
\[
C_{3}:=\Supp\left(\mathcal{T}or^{\cO_{X}}(\cO_{X_{\red}},\cM)\right).
\]
From \cite[Th.\ 22.3]{matsumura1989commutative}, its complement $X\setminus C_{3}$
is the locus where $\cM|_{X}$ is flat. Now the open subset 
\[
U:=X\setminus(C_{1}\cup C_{2}\cup C_{3})
\]
is the largest open subset such that 
\begin{itemize}
\item $\Supp(\cM|_{U})=U$,
\item for every $x\in U$, $\cM_{x}$ is generated by one element as an
$\cO_{X,x}$-module,
\item $\cM|_{U}$ is a flat $\cO_{U}$-module.
\end{itemize}
Therefore, $U$ is the desired open subset. We can compute $C_{1}$,
$C_{2}$ and $C_{3}$ by some algorithms, for example, ones implemented
to Macaulay2 \cite{graysonmacaulay2}. 
\end{proof}
\begin{rem}
\label{rem:closure}Let $X=\Spec R$ and let $C=V(f_{1},\dots,f_{n})\subset X$
be a closed subset. Then the closed subset $\overline{X\setminus C}$
is defined by the ideal 
\[
\bigcap_{i=1}^{n}\Ker(R\to R_{f_{i}}).
\]
Indeed, $X\setminus C$ is covered by the affine open subsets $U_{i}=\{f_{i}\ne0\}$
and we have $\overline{X\setminus C}=\bigcup\overline{U_{i}}$. Each
$\overline{U_{i}}$ is defined by the ideal $\Ker(R\to R_{f_{i}})$. 
\end{rem}

We conclude:
\begin{prop}
For each Hilbert polynomial $P$, we can explicitly compute open subschemes
$\ulHom_{P}(X,Y)$ and $\ulIso_{P}(X,Y)$ of $\Hilb_{P}(X\times Y)$
by means of explicit presentation of open subsets $\ulHom_{P}(X,Y)_{K}$
and $\ulIso_{P}(X,Y)_{K}$ of $\Hilb_{P}(X\times Y)_{K}$ for each
$K$. 
\end{prop}

\begin{rem}
\label{rem:2nd proof semidecidability}Using Iso schemes, we can give
an alternative proof of the semi-decidability of the isomorphism,
which was proved in Section \ref{sec:Semidecidability Iso Prob}.
We enumerate all the Hilbert polynomials as $P_{i}$, $i\in\ZZ_{>0}$.
For each $i>0$, we compute the Iso scheme $\ulIso_{P_{i}}(X,Y)$.
We stop if we get a non-empty Iso scheme. 
\end{rem}

\section{\label{sec:One-dimensional-schemes}One-dimensional schemes}

In this section, we show that the isomorphism problem for one-dimensional
projective schemes and the one for one-dimensional reduced quasi-projective
schemes are decidable. This generalizes the known case of smooth irreducible
curves (\cite[Lem.\ 5.1]{baker2005finiteness} for the case of genus
$\ne1$ and the MathOverflow thread mentioned in Introduction for
elliptic curves). We need the following version of the Riemann-Roch
formula for one-dimensional projective schemes.
\begin{prop}[{\cite[Exercise 18.4.S]{vakilfoundations}}]
\label{prop:RR}Let $X$ be a one-dimensional projective scheme and
let $X_{i}$, $1\le i\le l$, be its one-dimensional irreducible components
given with reduced structure and let $\eta_{i}$ be the generic point
of $X_{i}$. Let $\cL$ be an invertible sheaf on $X$ and let $\cF$
be a coherent sheaf on $X$. Then 
\begin{equation}
\chi(\cF\otimes\cL)-\chi(\cF)=\sum_{i=1}^{l}\length(\cF_{\eta_{i}})\deg(\cL|_{X_{i}}).\label{eq:RR}
\end{equation}
Here $\length(\cF_{\eta_{i}})$ is the length of $\cF_{\eta_{i}}$
as an $\cO_{X,\eta_{i}}$-module. In particular, 
\begin{equation}
\chi(\cL)-\chi(\cO)=\sum_{i=1}^{l}\length(\cO_{X,\eta_{i}})\deg(\cL|_{X_{i}}).\label{eq:RR-1}
\end{equation}
\end{prop}

\begin{proof}
The outline of the proof is written in \cite{vakilfoundations}. For
the sake of completeness, we write it down in more details. We first
observe that zero-dimensional connected components of $X$ do not
contribute to either side of (\ref{eq:RR}). Therefore we may suppose
that $X$ has only one-dimensional irreducible components, that is,
$X=\bigcup_{i=1}^{l}X_{i}$. Note also that both sides of (\ref{eq:RR})
are also additive for short exact sequences; if $v(\cF)$ denotes
either side of the equality, for a short exact sequence of coherent
$\cO_{X}$-modules,
\[
0\to\cF_{1}\to\cF_{2}\to\cF_{3}\to0,
\]
we have $v(\cF_{2})=v(\cF_{1})+v(\cF_{3})$. This implies that for
a filtration of coherent sheaves,
\[
\cF=\cF_{0}\supset\cF_{1}\supset\cdots\supset\cF_{n}=0,
\]
we have 
\begin{equation}
v(\cF)=\sum_{i}v(\cF_{i}/\cF_{i+1}).\label{eq:v add}
\end{equation}
Let $\cI\subset\cO_{X}$ be the defining ideal sheaf of the associated
reduced scheme $X_{\red}$ of $X$, which is necessarily nilpotent.
We apply equality (\ref{eq:v add}) to the filtration
\[
\cF\supset\cI\cF\supset\cI^{2}\cF\supset\cdots\supset\cI^{n}\cF=0.
\]
Thus it suffices to show (\ref{eq:RR}) for sheaves $\cI^{i}\cF/\cI^{i+1}\cF$.
Since they are $\cO_{X_{\red}}$-modules, in turn, it suffices to
show the proposition in the case where $X$ is reduced. Let us now
write $\cL=\cO_{X}(\sum_{j=1}^{m}n_{j}p_{j})$, where $n_{i}$ are
integers and $p_{i}$ are closed points of $X$ at which $X$ is smooth
and $\cF$ is locally free. We prove (\ref{eq:RR}) in this situation
by induction on $n:=\sum_{j=1}^{m}|n_{j}|$. If $n=0$, this is obvious.
If $n=1$, then $\cL$ is either $\cO_{X}(p)$ or $\cO_{X}(-p)$.
In the latter case, from the exact sequence
\[
0\to\cF\otimes\cL\to\cF\to\cF|_{p}\to0,
\]
we have
\[
\chi(\cF\otimes\cL)-\chi(\cF)=-\chi(\cF|_{p})=-\length(\cF_{\eta_{i}}),
\]
where $i$ is such that $p\in X_{i}$. If $\cL=\cO_{X}(p)$ and if
we put $\cF':=\cF\otimes\cL$, then 
\[
\chi(\cF\otimes\cL)-\chi(\cF)=-\left(\chi(\cF'\otimes\cO_{X}(-p))-\chi(\cF')\right)=\length(\cF_{\eta_{i}}).
\]
Thus (\ref{eq:RR}) holds when $n=1$. For general $n\ge1$, if we
write $\cL=\cL'\otimes\cL''$ with $\cL'$ and $\cL''$ having smaller
$n$, we have
\begin{align*}
\chi(\cF\otimes\cL)-\chi(\cF) & =\left(\chi((\cF\otimes\cL')\otimes\cL'')-\chi(\cF\otimes\cL')\right)+\left(\chi(\cF\otimes\cL')-\chi(\cF)\right)\\
 & =\sum_{i=1}^{l}\length(\cF_{\eta_{i}})\deg(\cL'|_{X_{i}})+\sum_{i=1}^{l}\length(\cF_{\eta_{i}})\deg(\cL''|_{X_{i}})\\
 & =\sum_{i=1}^{l}\length(\cF_{\eta_{i}})\deg(\cL|_{X_{i}}).
\end{align*}
\end{proof}
Consider two one-dimensional projective schemes $X\subset\PP^{m-1}$
and $Y\subset\PP^{n-1}$, which have very ample invertible sheaves
$\cL$ and $\cM$ corresponding to the given embeddings to projective
spaces respectively. Suppose that there exists an isomorphism $f\colon X\to Y$.
We will bound possibilities for the Euler characteristic of $\cL\otimes f^{*}\cM$
without using data of $f$. 
\begin{cor}
Let $X_{i}$, $1\le i\le l$ and $Y_{j}$, $1\le j\le m$ be the one-dimensional
irreducible components of $X$ and $Y$ respectively. We give them
with reduced structure. Let $\eta_{i}$ be the generic point of $X_{i}$.
Let $d:=\sum_{j=1}^{m}\deg(\cM|_{Y_{j}})$. Then there exists a partition
of $d$ into positive integers, $d=\sum_{i=1}^{l}d_{i}$, such that
\[
\chi(\cL\otimes f^{*}\cM)=\chi(\cO_{X})+\sum_{i=1}^{l}\length(\cO_{X,\eta_{i}})\left(\deg(\cL|_{X_{i}})+d_{i}\right).
\]
\end{cor}

\begin{proof}
We put $d_{i}:=\deg(f^{*}\cM|_{X_{i}})$ and apply the second equality
in Proposition \ref{prop:RR} with $\cL\otimes f^{*}\cM$ in place
of $\cL$.
\end{proof}
The Hilbert polynomial of $f$ (see Definition \ref{def: Hilb poly of an iso})
is equal to the Hilbert polynomial of $X$ with respect to the very
ample sheaf $\cL\otimes f^{*}\cM$. It is the polynomial $P(t)$ of
degree at most one such that 
\[
P(0)=\chi(\cO_{X})\text{ and }P(1)=\chi(\cL\otimes f^{*}\cM).
\]

\begin{thm}
\label{thm:1-dim proj}The isomorphism problem for one-dimensional
projective schemes is decidable.
\end{thm}

\begin{proof}
Let $X\subset\PP^{m-1}$ and $Y\subset\PP^{n-1}$ be one-dimensional
projective schemes. We compute their one-dimensional irreducible components
with reduced structure $X_{i}$, $1\le i\le l$ and $Y_{j}$, $1\le j\le m$
respectively; there exist algorithms to compute associated reduced
schemes and (geometric) irreducible components (see \cite{eisenbud1992directmethods,chistov1986algorithm}).
Then we compute $d=\sum_{j=1}^{m}\deg(\cM|_{Y_{j}})$. 

For each partition 
\[
\lambda\colon d=d_{1}+\cdots+d_{l}
\]
 of $d$ into $l$ positive integers, we compute
\[
e_{\lambda}:=\chi(\cO_{X})+\sum_{i=1}^{l}\length(\cO_{X,\eta_{i}})\left(\deg(\cL|_{X_{i}})+d_{i}\right)
\]
and define the polynomial
\[
P_{\lambda}(t):=(e_{\lambda}-\chi(\cO_{X}))t+\chi(\cO_{X}).
\]
The polynomials $P_{\lambda}$ are the only potential Hilbert polynomials
for an isomorphism $f\colon X\to Y$ if any. For each $\lambda$,
we compute $\ulIso_{P_{\lambda}}(X,Y)$. If one of them is non-empty,
then $X$ and $Y$ are isomorphic. Otherwise, they are not isomorphic. 
\end{proof}
\begin{thm}
\label{thm: iso one-dim quasi-proj}The isomorphism problem for one-dimensional
reduced quasi-projective schemes is decidable. Here we suppose that
each quasi-projective scheme $X$ is given an embedding $X\hookrightarrow\PP^{m-1}$
and represented by two projective schemes $\overline{X}\subset\PP^{m-1}$,
the closure of $X$ in $\PP^{m-1}$, and $\overline{X}\setminus X\subset\PP^{m-1}$. 
\end{thm}

\begin{proof}
Let $X\subset\PP^{m-1}$ and $Y\subset\PP^{n-1}$ be quasi-projective
one-dimensional reduced schemes and let $\overline{X}\subset\PP^{m-1}$
and $\overline{Y}\subset\PP^{n-1}$ be their closures respectively.
If $\overline{X}$ is singular at some point of $\overline{X}\setminus X$,
then we resolve this singularity by repeating blowups. Note that a
blowup of $\PP^{m-1}$ at a point is a closed subvariety of $\PP^{m-1}\times\PP^{m-2}$
and hence one of $\PP^{m(m-1)-1}$ by the Segre embedding. Thus a
blowup of $\overline{X}$ at a point has an embedding into $\PP^{m(m-1)-1}$,
which can be explicitly computed. Therefore we can replace the embedding
$X\subset\PP^{m-1}$ so that $\overline{X}$ becomes smooth at every
point of $\overline{X}\setminus X$. Similarly for $Y$. If $\overline{X}\setminus X$
and $\overline{Y}\setminus Y$ have different numbers of points, then
$X$ and $Y$ are not isomorphic. Thus, we may suppose that they have
the same number of points. Now $X$ and $Y$ are isomorphic if and
only if there exists an isomorphism $f\colon\overline{X}\to\overline{Y}$
such that $\overline{Y}\setminus Y\subset f(\overline{X}\setminus X)$.
From the assumption which we just put, $\overline{Y}\setminus Y\subset f(\overline{X}\setminus X)$
implies $\overline{Y}\setminus Y=f(\overline{X}\setminus X)$. Following
the algorithm described in the proof of Theorem \ref{thm:1-dim proj},
we can compute the Iso scheme
\[
\ulIso(\overline{X},\overline{Y})\left(=\bigcup_{\lambda}\ulIso_{P_{\lambda}}(\overline{X},\overline{Y})\right),
\]
where $\lambda$ runs over partitions of a positive integer $d$ as
in the proof of Theorem \ref{thm:1-dim proj}. In particular, since
there are only finitely many partitions, we can algorithmically compute
this Iso scheme. Recall that the Iso scheme is by definition a subscheme
of the Hilbert scheme $\Hilb(\overline{X}\times\overline{Y})$. Let
\[
\cU\subset\ulIso(\overline{X},\overline{Y})\times\overline{X}\times\overline{Y}
\]
be the universal family. Let $A\subset\cU$ be the preimage of $\overline{X}\setminus X$
by the projection $\cU\to\overline{X}$. Let us write $\overline{Y}\setminus Y=\{y_{1},\dots,y_{n}\}$
and let $B_{1},\dots,B_{n}\subset\cU$ be the preimages of $y_{1},\dots,y_{n}$
by the projection $\cU\to\overline{Y}$ respectively. Let $\pi\colon\cU\to\ulIso(\overline{X},\overline{Y})$
be the projection. For each $i$, we claim that
\[
\pi(A\cap B_{i})=\{[f]\in\ulIso(\overline{X},\overline{Y})\mid y_{i}\in f(\overline{X}\setminus X)\}.
\]
Indeed, $\pi^{-1}([f])$ is identical to the graph $\Gamma_{f}\subset\overline{X}\times\overline{Y}$
and we have
\[
\pi^{-1}([f])\cap A\cap B_{i}=\{(x,y_{i})\in\overline{X}\times\overline{Y}\mid x\in\overline{X}\setminus X\}.
\]
The last set is non-empty if and only if $y_{i}\in f(\overline{X}\setminus X)$.
This shows the above claim. 

Thus $\bigcap_{i=1}^{n}\pi(A\cap B_{i})$ is exactly the locus of
isomorphisms $f\colon\overline{X}\to\overline{Y}$ with $\overline{Y}\setminus Y\subset f(\overline{X}\setminus X)$.
We compute this closed subset $\bigcap_{i=1}^{n}\pi(A\cap B_{i})$
and check whether this is empty. The given quasi-projective schemes
$X$ and $Y$ are isomorphic if and only if this is not empty. 
\end{proof}

\section{Varieties with a big canonical sheaf or a big anti-canonical sheaf\label{sec: genetal type Fano}}

In this section, we show the decidability of the isomorphism problem
for varieties as in the title, generalizing the case of general type
solved by Totaro (see \cite[Rem.\ 12.3]{poonen2014undecidable}).
The key ingredients are computation of Iso schemes and the Kodaira
vanishing theorem.

\begin{thm}
\label{thm:iso prob big}Let $X$ and $Y$ be smooth irreducible projective
varieties. Suppose that either $\omega_{X}$ or $\omega_{X}^{-1}$
is big. Then we can algorithmically decide whether $X$ and $Y$ are
isomorphic. 
\end{thm}

\begin{proof}
It is enough to consider the case where $X$ and $Y$ have the same
dimension $d>0$. We may also suppose that either both $\omega_{X}$
and $\omega_{Y}$ are big or both $\omega_{X}^{-1}$ and $\omega_{Y}^{-1}$
are big. For, otherwise, $X$ and $Y$ are not isomorphic. We denote
these big sheaves by $\cB_{X}$ and $\cB_{Y}$ respectively. Note
that we can algorithmically decide which of $\omega_{X}$ and $\omega_{X}^{-1}$
is big by checking the birationality of maps $\Phi_{\omega_{X}^{\otimes n}}$,
$n\in\ZZ$ in turn (see Section \ref{subsec:Nefness-and-bigness}).
Let $\cL$ and $\cM$ be the very ample invertible sheaves on $X$
and $Y$ respectively corresponding to the given embeddings into projective
spaces. We compute the least positive integer $e$ such that $\cL^{\otimes e}\otimes\omega_{X}^{-1}$
is ample (see Section \ref{subsec:Ampleness}). We replace $\cL$
with $\cL^{\otimes e}$, which amounts to replacing the embedding
$X\hookrightarrow\PP^{m-1}$ with the one obtained by the $e$-uple
Veronese embedding $\PP^{m-1}\hookrightarrow\PP^{\binom{m-1+e}{m-1}-1}$.
Now $\cL\otimes\omega_{X}^{-1}$ is ample. In this situation, we will
algorithmically output finitely many polynomials $Q_{1},\dots,Q_{c}$
such that the Hilbert polynomial of every isomorphism $f\colon X\to Y$
(if any) is one of them. To do so, we first note that for every positive
integer $l$, $(\cL\otimes f^{*}\cM)^{\otimes l}\otimes\omega_{X}^{-1}$
is ample. From the Kodaira vanishing, we have 

\[
\H^{i}(X,(\cL\otimes f^{*}\cM)^{\otimes l})=0\quad(i>0).
\]
Hence the Hilbert polynomial $P_{f}$ of $f$ satisfies
\[
P_{f}(l)=\chi(X,(\cL\otimes f^{*}\cM)^{\otimes l})=h^{0}(X,(\cL\otimes f^{*}\cM)^{\otimes l}).
\]
Here $h^{0}$ means the dimension of $\H^{0}$. Then we compute the
least positive integer $q$ such that $\cB_{X}^{\otimes q}\otimes\cL^{-1}$
are $\cB_{Y}^{\otimes q}\otimes\cM^{-1}$ are both effective. Then,
there exists an injection
\[
0\to(\cL\otimes f^{*}\cM)^{\otimes l}\to\cB_{X}^{\otimes2ql},
\]
which implies
\[
P_{f}(l)=h^{0}(X,(\cL\otimes f^{*}\cM)^{\otimes l})\le h^{0}(X,\cB_{X}^{\otimes2ql}).
\]
Note that the obtained upper bound $h^{0}(X,\cB_{X}^{\otimes2ql})$
of $P_{f}(l)$ is independent of the isomorphism $f\colon Y\to X$.
In general, a polynomial 
\[
h(t)=a_{d}t^{d}+\cdots+a_{0}
\]
of degree $d$ is determined by its values at $d+1$ distinct points
$t_{0},\dots,t_{d}$, by Lagrange interpolation. We compute $h^{0}(X,\cB_{X}^{\otimes2ql})$
for $1\le l\le d+1$. For each tuple $\lambda=(\lambda_{1},\dots,\lambda_{d+1})$
of nonnegative integers with $\lambda_{i}\le h^{0}(X,\cB_{X}^{\otimes2qi})$,
we compute the polynomial $Q_{\lambda}(t)$ such that $Q_{\lambda}(i)=\lambda_{i}$.
Thus obtained finitely many polynomials $Q_{\lambda}$ are the desired
ones. For each $\lambda$, we compute $\ulIso_{Q_{\lambda}}(X,Y)$
and check whether it is empty or not. If one of them is not empty,
then $X$ and $Y$ are isomorphic. Otherwise, they are not isomorphic. 
\end{proof}

\section{\label{sec:Computing-intersection-numbers}Computing intersection
numbers}

The aim of this section is to prove the following proposition. 
\begin{prop}[{\cite[Lem.\ 8.7]{poonen2015computing}}]
\label{prop: intersection}For a smooth irreducible projective variety
$X$, an irreducible closed subset $Z\subset X$ of codimension $c$
and an invertible sheaf $\cL$ on $X$, we can algorithmically compute
the intersection number $Z\cdot\cL^{\dim Z}\in\ZZ$. 
\end{prop}

This will be used in Section \ref{sec:Positivity} to discuss decidability
of various positivity properties of invertible sheaves. The proof
in \cite{poonen2015computing} uses étale cohomology (in fact, its
authors considered, more generally, the intersection number of cycles
of complementary dimensions). We give an alternative proof using Simpson's
algorithm \cite[Section 2.5]{simpson2008algebraic} to compute singular
cohomology. We only consider the case where the ambient variety $X$
is smooth, as Simpson's algorithm is valid only for smooth varieties.
According to his algorithm, for a smooth projective variety $X\subset\PP^{r-1}$,
we can compute a finite simplicial complex $\cH(X)$ in $\PP^{r-1}(\CC)$
which is homotopy equivalent to $X(\CC)$. In particular, we can compute
the singular homology $\H_{i}(X(\CC),\ZZ)$ and cohomology $\H^{i}(X(\CC),\ZZ)$
using $\cH(X)$. Their elements are represented by simplicial $i$-chains
and $i$-cochains on $\cH(X)$ respectively. 

When we have a morphism $f\colon Z\to X$ of smooth projective varieties,
denoting its graph by $\Gamma_{f}$, we can compute maps of simplicial
complexes
\[
\cH(Z)\leftarrow\cH(\Gamma_{f})\to\cH(X)
\]
and the induced map
\[
\H_{i}(Z(\CC),\ZZ)\cong\H_{i}(\Gamma_{f}(\CC),\ZZ)\xrightarrow{f_{*}}\H_{i}(X(\CC),\ZZ).
\]
If $X$ and $Z$ are irreducible and have dimensions $d$ and $p$
respectively and if $f$ is generically finite onto the image, then
the cycle class $[f]\in\H^{2d-2p}(X(\CC),\ZZ)$ of $f$ is computed
to be the element corresponding to $f_{*}([Z])\in\H_{2p}(X(\CC),\ZZ)$
via the Poincaré duality
\[
\H^{2d-2p}(X(\CC),\ZZ)\cong\H_{2p}(X(\CC),\ZZ).
\]
Here $[Z]\in\H_{2p}(Z(\CC),\ZZ)$ denotes the fundamental class of
$Z$. When $Z\subset X$ is a (possibly singular) irreducible closed
subvariety of dimension $p$, then we can algorithmically construct
a resolution of singularities $f\colon\widetilde{Z}\to Z\subset X$
(see \cite{villamayor1989constructiveness,villamayoru.1992patching,bierstone1991asimple,bierstone1997canonical,bodnar2000acomputer})
and define the cycle class $[Z]\in\H^{2d-2p}(X(\CC),\ZZ)$ to be $[f]$. 

For an invertible sheaf $\cL$ on $X$, we can compute a divisor $D$
on $X$ such that $\cL\cong\cO_{X}(D)$, for example, by an algorithm
given in \cite[Section 3]{schwede2018divisor}. If we write $D=\sum_{i=1}^{n}a_{i}D_{i}$
with $D_{i}$ prime divisors and $a_{i}$ integers, then the cohomology
class $[\cL]$ of $\cL$ is defined to be $[D]=\sum_{i=1}^{n}a_{i}[D_{i}]$. 

We can also compute the cup product 
\[
\H^{i}(X(\CC),\ZZ)\times\H^{j}(X(\CC),\ZZ)\to\H^{i+j}(X(\CC),\ZZ)
\]
again by using the representation of $\H^{i}(X(\CC),\ZZ)$ in terms
of the simplicial complex $\cH(X)$. In summary, in the situation
of Proposition \ref{prop: intersection}, we can algorithmically compute
elements $[Z],[\cL]\in\H^{2d-2c}(X(\CC),\ZZ)$ as represented by explicit
cochains on $\cH(X)$ and compute the product $[Z][\cL]^{\dim Z}\in\H^{2d}(X(\CC),\ZZ)$
with respect to the cup product. The desired intersection number $Z\cdot\cL^{\dim Z}\in\ZZ$
is then computed as the integer $n$ such that $[Z][\cL]^{\dim Z}=n[\pt]$,
where $[\pt]$ is the cycle class of a point of $X(\CC)$. This completes
the proof of Proposition \ref{prop: intersection}.

\section{\label{sec:Positivity}Positivity of invertible sheaves}

Positivity properties of invertible sheaves, such as ample, big, and
nef, are closely related to the isomorphism problem. In this section,
we discuss the decidability problem of these properties. We also show
that, for a smooth variety whose Picard number can be computable,
we can approximate its nef cone and pseudo-effective cone with arbitrary
precision. 

\subsection{Global generation}
\begin{prop}
\label{prop: global gen}For a projective variety $X$ and a coherent
sheaf $\cL$ on it, we can algorithmically check whether it is globally
generated. (We do not assume that $\cL$ is invertible, although it
is the case of our main interest.)
\end{prop}

\begin{proof}
Let $R$ be the homogeneous coordinate ring of $X$ and let $L$ be
the given finitely generated graded $R$-module, which defines $\cL$.
We can compute the graded $R$-module $L':=\bigoplus_{v\ge0}\H^{0}(X,\widetilde{L}(v))$
as the Hom module $\Hom_{R}(R_{\ge r},L)_{\ge0}$ for some sufficiently
large $r$, see Theorem 8.2 of Chapter 8 by Eisenbud in the book \cite{vasconcelos1998computational}.
Let 
\[
F_{1}\to F_{0}\to L'\to0
\]
be the obtained minimal free presentation of $L'$. Here the arrows
are degree-preserving $R$-linear maps and the free module $F_{0}$
is written as 
\[
F_{0}=\bigoplus_{i=1}^{c}Ry_{i}
\]
with homogeneous generators $y_{i}$ with 
\[
0\le\deg(y_{1})\le\cdots\le\deg(y_{c}).
\]
Let $y_{1},\dots,y_{n}$ ($n\le c$) be the ones of degree $0$, which
are regarded as a basis of $\H^{0}(X,\widetilde{L})$, and let 
\[
F_{0}':=\bigoplus_{i=1}^{n}Ry_{i}.
\]
The derived map $F_{0}'\to L'$ induces the map of sheaves,
\[
\cO_{X}\otimes\H^{0}(X,\widetilde{L})\to\widetilde{L}.
\]
The sheaf $\widetilde{L}$ is globally generated if and only if this
map is surjective. We can check the latter condition, for example,
by computing the support of the $\cO_{X}$-module corresponding to
the graded $R$-module $\Coker(F_{0}'\to L')$ and see whether it
is empty.
\end{proof}

\subsection{Very ampleness\label{subsec:Very-ampleness}}
\begin{prop}
For a projective variety $X$ and an invertible sheaf $\cL$ on it,
we can algorithmically check whether it is very ample. 
\end{prop}

\begin{proof}
We first check the global generation of $\cL$. If $\cL$ is not globally
generated, then $\cL$ is not very ample. Suppose that $\cL$ is globally
generated. We then compute the morphism $\Phi_{\cL}\colon X\to\PP^{n-1}$
associated to $\cL$, where $n=\dim\H^{0}(X,\cL)$. Let $L$ be the
given graded $R$-module defining $\cL$ and we construct a map of
$R$-modules, $F_{0}'\to L'$, as in the proof of Proposition \ref{prop: global gen}.
Let $M$ be the image of this map, which defines the same sheaf on
$X$ as $L'$ and $L$ do. We have a free presentation 
\[
F_{1}'\to F_{0}'\to M\to0.
\]
From \cite[Prop.\ A2.2]{eisenbud1995commutative}, this induces the
exact sequence
\[
\Sym_{R}(F_{0}')\otimes_{R}F_{1}'\to\Sym_{R}(F_{0}')\to\Sym_{R}(M)\to0.
\]
If the basis of $F_{1}'$ maps to 
\[
g_{i}=\sum_{j=1}^{n}g_{ij}y_{j}\quad(i=1,\dots,r)
\]
with $r$ denoting the rank of $F_{1}'$, then 
\begin{align*}
\Sym_{R}(M) & =R[y_{1},\dots,y_{n}]/(g_{1},\dots,g_{r})\\
 & =k[x_{1},\dots,x_{m},y_{1},\dots,y_{n}]/(f_{1},\dots,f_{l},g_{1},\dots,g_{r}),
\end{align*}
which is bi-graded. This defines a closed subscheme $\Gamma\subset\PP^{m-1}\times\PP^{n-1}$.
The projection $\Gamma\to\PP^{m-1}$ is an isomorphism onto $X$.
In other words, $\Gamma$ is the graph of a morphism $X\to\PP^{n-1}$.
The last morphism is the morphism $\Phi_{\widetilde{L}}$ associated
to the globally generated invertible sheaf $\widetilde{L}$. Now we
can compute the image $Y:=\Phi_{\widetilde{L}}(X)$ of $\Phi_{\widetilde{L}}$
by projective elimination (see Remark \ref{rem:proj elim}). From
Lemma \ref{lem:iso decidable}, we can check whether the morphism
$X\to Y$, which corresponds to $\Gamma\subset\PP^{m-1}\times\PP^{n-1}$,
is an isomorphism. Our invertible sheaf $\cL$ is very ample if and
only if the last morphism is an isomorphism.
\end{proof}
\begin{rem}[Projective elimination]
\label{rem:proj elim}Suppose that a closed subscheme $V\subset\PP^{m-1}\times\PP^{n-1}$
is defined by a bi-homogeneous ideal $I\subset k[x_{1},\dots,x_{m},y_{1},\dots,y_{n}]$.
Then the scheme-theoretic image $p_{2}(V)$ by the second projection
is defined by the ideal
\[
(I:(x_{1},\dots,x_{m})^{\infty})\cap k[y_{1},\dots,y_{n}].
\]
At the set-theoretic level, this is written \cite[page 503]{greuel2008asingular};
the closed subset $p_{2}(V)$ is the zero set of the last homogeneous
ideal. The scheme-theoretic version follows from \cite[Lemma A.7.9]{greuel2008asingular}. 
\end{rem}

\subsection{Ampleness\label{subsec:Ampleness}}
\begin{prop}
\label{prop: ample}We can algorithmically decide whether an invertible
sheaf on a smooth irreducible projective variety is ample. 
\end{prop}

\begin{proof}
Let $X$ be a projective scheme and let $\cL$ be an invertible sheaf
on $X$. From an effective version of Matsusaka's Big Theorem \cite{siu1993aneffective},
there exists a positive integer $m(\cL)$ explicitly determined by
$\cL^{\dim X}$, $\omega_{X}\cdot\cL^{\dim X-1}$, and $\dim X$ such
that $\cL$ is ample if and only if $\cL^{\otimes m(\cL)}$ is very
ample. Thus, we only need to compute the number $m(\cL)$ and check
whether $\cL^{\otimes m(\cL)}$ is very ample. 
\end{proof}

\subsection{\label{subsec:Nefness-and-bigness}Nefness, bigness and pseudo-effectivity}

These properties in the title of invertible sheaves are all positivity
properties in some sense, which are weaker than ampleness, and play
important roles in birational geometry. In what follows, we restrict
ourselves to the case where the ambient scheme is irreducible and
smooth, unless otherwise noted. 
\begin{defn}
Let $X$ be a smooth irreducible projective variety and let $\cL$
be an invertible sheaf on $X$. The $\cL$ is \emph{big} if for some
integer $n>0$, the rational map associated to $\cL^{n}$,
\[
\Phi_{\cL^{n}}\colon X\dashrightarrow\PP^{m},
\]
is birational onto the image. The $\cL$ is \emph{nef }if for every
irreducible curve $C\subset X$, we have $C\cdot\cL\ge0$. The $\cL$
is \emph{pseudo-effective }if its class in $\NS(X)\otimes\RR$ is
the limit of classes of effective divisors. 
\end{defn}

Note that it is easy to check whether $\cL$ is effective (that is,
isomorphic to $\cO_{X}(D)$ for some effective divisor $D$) by computing
the cohomology group $\H^{0}(X,\cL)$. For each $n$, we can algorithmically
check whether the map $\Phi_{\cL^{n}}$ is birational, see \cite{simis2004cremona,doria2012acharacteristicfree}.
See also \cite{bott2019rationalmaps} for implementation of such an
algorithm. Thus, bigness of an invertible sheaf is semi-decidable.
On the other hand, \emph{not} being nef is a semi-decidable property.
Indeed, we enumerate irreducible curves in $X$ as $C_{1},C_{2},\dots$
and for each $i$, we compute the intersection number $C_{i}\cdot\cL$,
until we get a negative intersection number. As for pseudo-effectivity,
we have the following theorem \cite[0.2 Theorem]{boucksom2013thepseudoeffective}:
$\cL$ is pseudo-effective if and only if $\cL\cdot C\ge0$ for every
irreducible curve $C\subset X$ which moves in a family covering $X$.
Using this, we can prove:
\begin{prop}
\label{prop: not pseudo-eff}Not being pseudo-effective is a semi-decidable
property. 
\end{prop}

\begin{proof}
We enumerate all irreducible curves on $X$ as $C_{1},C_{2},\dots$.
Consider the following algorithm:
\begin{enumerate}
\item Put $n=1$.
\item Check whether  $C_{n}$ moves in a family covering $X$ as follows.
We first compute the Hilbert polynomial $P_{n}$ of $C_{n}$ and then
compute the connected component $W$ of $\Hilb_{P_{i}}(X)$ containing
$[C_{n}]$. We then check whether  the universal family $\cU_{W}$
on $W$ maps onto $X$; $C_{n}$ is a movable curve if and only if
this is the case. When $C_{n}$ is movable, we compute the intersection
number $C_{n}\cdot D$ and stop the algorithm if $C_{n}\cdot D<0$.
\item Put $n=n+1$ and go back to (2).
\end{enumerate}
From \cite[0.2 Theorem]{boucksom2013thepseudoeffective}, this algorithm
stops after finitely many steps if and only if $D$ is not pseudo-effective.
The proposition follows.
\end{proof}
In summary, the following properties of invertible sheaves on a smooth
irreducible projective variety are semi-decidable:
\begin{enumerate}
\item Being big.
\item Not being nef.
\item Not being pseudo-effective.
\end{enumerate}
It is now quite natural to ask:
\begin{problem}
Are the three properties, big, nef and pseudo-effective, decidable?
\end{problem}

\begin{rem}
\label{rem:nef and big}Note that when the given invertible sheaf
is known to be nef, then we only need to compute the intersection
number $\cL^{\dim X}$ to check whether  $\cL$ is big (see \cite[Theorem 2.2.16]{lazarsfeld2004positivity2}).
\end{rem}

In the case of the canonical sheaf $\omega_{X}$, the most important
invetible sheaf, we may take advantage of the following conjectures:
\begin{conjecture}[The abundance conjecture]
For a smooth projective variety $X$, the canonical sheaf $\omega_{X}$
is nef if and only if it is semi-ample (that is, $\cL^{n}$ is globally
generated for some $n>0$). 
\end{conjecture}

\begin{conjecture}[The non-vanishing conjecture]
For a smooth projective variety $X$, the canonical sheaf $\omega_{X}$
is pseudo-effective if and only if it is $\QQ$-linearly equivalent
to an effective $\QQ$-divisor. 
\end{conjecture}

The abundunce conjecture is recognized as one of the most important
conjectures in the minimal model program. The importance of the non-vanishing
conjecture was pinned down by Birkar \cite{birkar2011onexistence}.
The above form of the non-vanishing conjecture is slightly different
from the one considered by Birkar. However Hashizume \cite{hashizume2018onthe}
proved that they are equivalent. 
\begin{prop}
Let $X$ be a smooth irreducible projective variety.
\begin{enumerate}
\item If the abundance conjecture holds for $X$, then the nefness of $\omega_{X}$
is decidable. 
\item If the weak nonvanishing conjecture holds for $X$, then the pseudo-effectivity
of $\omega_{X}$ is decidable.
\end{enumerate}
\end{prop}

\begin{proof}
(1) We first note that we can compute the canonical sheaf \cite[Section 5.6]{stillmancomputing}.
We enumerate all irreducible curves in $X$ as $C_{1},C_{2},\dots$.
Consider the following algorithm:
\begin{enumerate}
\item Put $n=1$.
\item We check whether  $\omega_{X}^{\otimes n}$ is globally generated.
If this is the case, then stop the algorithm and output True.
\item We check whether  $C_{n}\cdot\omega_{X}<0$. If this is the case,
then stop the algorithm and output False. 
\item Put $n=n+1$ and go to (2). 
\end{enumerate}
If the abundance conjecture holds, then this algorithm always stops
after finitely many steps and outputs True if $\omega_{X}$ is nef
and False if $\omega_{X}$ is not nef. 

(2) Effective $\QQ$-divisors on $X$ are enumerable. For each of
them, we can compute its class in $\H^{2}(X(\CC),\QQ)$ by the method
explained in Section \ref{sec:Computing-intersection-numbers} and
check whether it coincides with the class of $\omega_{X}$. From the
non-vanishing conjecture, we see that the pseudo-effectivity of $\omega_{X}$
is semi-decidable. Combining this with Proposition \ref{prop: not pseudo-eff}
shows the assertion.
\end{proof}
\begin{rem}
\label{rem: bigness of omega}If $\dim X\le3$, then the bigness of
$\omega_{X}$ is also decidable. Indeed, from \cite{hacon2006boundedness,takayama2006pluricanonical,tsuji2006pluricanonical},
for each dimension $d$, there exists a positive integer $n_{d}$
such that for every smooth variety $X$ of general type and of dimension
$d$ and for every integer $n\ge n_{d}$, the rational map $\Phi_{\omega_{X}^{\otimes n}}$
is birational onto the image. Moreover, for $d\le3$, we can take
$n_{1}=3$, $n_{2}=5$ , $n_{3}=126$ (see \cite{bombiericanonical,chen2010explicit});
we only need to check whether $\Phi_{\omega_{X}^{\otimes n_{d}}}$
is birational onto the image. To generalize this argument to dimensions
$\ge4$, we need to compute $n_{d}$. 

If $\omega_{X}$ is nef, then we can check its bigness in any dimension,
see Remark \ref{rem:nef and big}. If $\omega_{X}$ is not nef, then
we may run the minimal model program. As an output of the program,
we would get a Mori fiber space or a minimal model birational to the
given variety $X$. In the former case, $\omega_{X}$ is not pseudo-effective,
in particular, not big. In the latter case, we can check the bigness
of $\omega_{X}$ by computing the intersection number $(\omega_{X})^{\dim X}$.
This strategy provides motivation for studying the following problem,
which would be important also on its own right: 
\end{rem}

\begin{problem}
Describe each step of the minimal model program as a strict algorithm,
starting from algorithmically finding a $\omega_{X}$-negative ray
of the cone of curves. 
\end{problem}

\subsection{\label{subsec:Approximating cones}Approximating nef and pseudo-effective
cones}

Let $\NS(X)_{\RR}:=\NS(X)\otimes\RR$ denote the Néron-Severi group
tensored with $\RR$, This is a finite-dimensional $\RR$-vector space
and its dimension $\rho(X)$ is called the Picard number of $X$.
The \emph{nef cone} of $X$, denoted by $\Nef(X)$, is the smallest
closed convex cone in $\NS(X)_{\RR}$ such that, for an invertible
sheaf $\cL$, the class $[\cL]$ belongs to it if and only if $\cL$
is nef. The \emph{pseudo-effective cone $\PEff(X)$ }is similarly
defined. The \emph{ample cone }and the \emph{big cone }are the interiors
of the nef cone and the pseudo-effective cone respectively. 

As we do not have an algorithm to decide whether a given invertible
sheaf is big/nef, we can not compute the cones $\PEff(X)$ and $\Nef(X)$
at least for now. However, if we know the value of the Picard number
$\rho(X)$, then we can approximate these cones with arbitrary precision.
Note that, if we know the value of $\rho(X)$, then we can compute
the subspace $\NS(X)_{\RR}\subset\H^{2}(X(\CC),\RR)$ by giving a
basis of it. To do so, we only need to compute classes $[D]\in\H^{2}(X(\CC),\RR)$
of divisors $D\subset X$, until we have enough to span a subspace
of dimension $\rho(X)$. Poonen, Testa and van Luijk \cite{poonen2015computing}
gave an algorithm to compute $\rho(X)$, assuming the Tate conjecture.
In particular, we can compute $\rho(X)$ if $X$ is a K3 surface.
\begin{prop}
\label{prop: approx cones}Let $X$ be a smooth irreducible projective
variety. Suppose that we know the value of $\rho(X)$. We fix a metric
on $\NS(X)_{\RR}$. Let $S\subset\NS(X)_{\RR}$ be the unit sphere
with center at the origin. Then, for any positive real number $\epsilon>0$,
we can algorithmically construct rational polyhedral convex cones
$A_{\epsilon}$ and $B_{\epsilon}$ such that $A_{\epsilon}\subset\Nef(X)\subset B_{\epsilon}$
and $B_{\epsilon}\cap S$ is contained in the $\epsilon$-neighborhood
of $A_{\epsilon}\cap S$. Similarly for $\PEff(X)$. 
\end{prop}

\begin{proof}
Let $\rho$ denote the Picard number of $X$. From Proposition \ref{prop: ample},
we can enumerate all the ample divisors on $X$ as $D_{1},D_{2},\dots$.
Let $A_{n}=\sum_{i=1}^{n}\RR_{\ge0}[D_{i}]$ be the convex cone generated
by $[D_{1}],\dots,[D_{n}]$. The closure of $\bigcup_{n\ge0}A_{n}$
is the nef cone $\Nef(X)$. In particular, each $A_{n}$ is a rational
convex polyhedral cone contained in $\Nef(X)$. We can also enumerate
the irreducible curves in $X$ as $C_{1},C_{2},\dots$. Let 
\begin{align*}
B_{n}^{c} & :=\bigcup_{i=1}^{n}\{x\in\NS(X)_{\RR}\mid x\cdot C_{i}<0\}\text{ and}\\
B_{n} & :=\NS(X)_{\RR}\setminus B_{n}^{c}=\bigcap_{i=1}^{n}\{x\in\NS(X)_{\RR}\mid x\cdot C_{i}\ge0\}.
\end{align*}
We see that $\bigcup_{n\ge0}B_{n}^{c}=\NS(X)_{\RR}\setminus\Nef(X)$.
Thus each $B_{n}$ is a rational convex polyhedral cone containing
$\Nef(X)$. It is also strongly convex (that is, it has a vertex at
the origin) for $n\gg0$. We have got two sequences $(A_{n})_{n}$
and $(B_{n})_{n}$ of rational convex polyhedral cones approximating
$\Nef(X)$ from inside and outside respectively. Therefore, for $n\gg0$,
$A_{n}$ and $B_{n}$ satisfy the desired condition. To see for which
value of $n$ this is the case, we first check whether  $B_{n}$ is
strongly convex. If this is the case, then for each vertex $w\in B_{n}\cap S$,
we check whether  every vertex $w\in B_{n}\cap S$ is contained in
the $\epsilon$-neighborhood of $A_{n}$. If this is the case, $B_{n}$
is contained in the $\epsilon$-neighborhood of $A_{n}$. This completes
the proof for the nef cone $\Nef(X)$.

As for the pseudo-effective cone $\PEff(X)$, we only need to replace
ample divisors with big divisors and irreducible curves with movable
irreducible curves. To enumerate movable irreducible curves, we can
use the algorithm in the proof of \ref{prop: not pseudo-eff}. To
enumerate big divisors, we can use the following algorithm: We first
enumerate all the divisors on $X$ as $D_{1},D_{2},\dots$.
\begin{enumerate}
\item Put $n=1$ and put $b=()$, the empty ordered tuple.
\item For each $i,j\le n$, if $\Phi_{i\cdot D_{j}}$ is a birational map
onto the image and if $D_{j}\notin b$, then append $D_{j}$ to $b$.
\item Put $n=n+1$ and go to (2).
\end{enumerate}
For every big divisor $D$ on $X$, the above algorithm appends $D$
to $b$ after finitely many steps. Thus, for every positive integer
$n$, we can algorithmically construct the $n$-th big divisor. (Thus,
big divisors on $X$ are listable. But this does not mean that the
bigness of each divisor is decidable.)
\end{proof}

\section{\label{sec: K3}K3 surfaces}

In this section, we discuss the isomorphism problem for K3 surfaces,
which would be natural as the next case to study after the one-dimensional
case and the case with $\omega_{X}$ or $\omega_{X}^{-1}$ big were
treated in Sections \ref{sec:One-dimensional-schemes} and \ref{sec: genetal type Fano}
respectively. The main result of this section is the decidability
of the isomorphism problem for K3 surfaces with an automorphism group
finite. 
\begin{prop}
\label{prop: compute Nef}Let $X$ be a K3 surface. If $\Aut(X)$
is finite, then we can compute the nef cone $\Nef(X)$ by giving finitely
many effective divisors $D_{1},\dots,D_{n}$ such that $\Nef(X)=\sum_{i=1}^{n}\RR_{\ge0}[D_{i}]$. 
\end{prop}

\begin{proof}
From \cite{poonen2015computing}, we can compute the Picard number
$\rho(X)$ and compute $\NS(X)_{\RR}$ as explained in Section \ref{subsec:Approximating cones}.
For an effective divisor $D$ on $X$, we can check whether  it is
nef; we check whether  $C\cdot D\ge0$ for every prime divisor $C$
contained in the support of $D$. Therefore we can enumerate all the
effective divisors as $D_{1},D_{2},\dots$ and all the nef and effective
divisors as $N_{1},N_{2},\dots$. For each $n$, let 
\begin{align*}
A_{n} & :=\sum_{i=1}^{n}\RR_{\ge0}[N_{i}],\\
B_{n} & :=\bigcap_{i=1}^{n}\{x\in\NS(X)_{\RR}\mid x\cdot D_{i}\ge0\}.
\end{align*}
These are rational polyhedral convex cones satisfying
\begin{equation}
A_{n}\subset\Nef(X)\subset B_{n}.\label{eq:incl}
\end{equation}
If $\Aut(X)$ is finite, then $\PEff(X)$ is a rational polyhedral
cone spanned by effective classes \cite{kovacs1994thecone}. It follows
that every point of $\NS(X)_{\QQ}\cap\PEff(X)$ is represented by
an effective $\QQ$-divisor. Since $\Nef(X)$ is the dual cone of
$\PEff(X)$, it is also rational polyhedral and spanned by finitely
many points of $\NS(X)_{\QQ}$. Since $\Nef(X)\subset\PEff(X)$, these
points are represented by nef and effective divisors after multiplied
with some positive integer. We conclude that $\Nef(X)$ is a rational
polyhedral cone spanned by nef and effective classes. Therefore, for
$n\gg0$, inclusions \vpageref{eq:incl} are equalities. For each
$n$, we compute $A_{n}$ and $B_{n}$ and check whether  $A_{n}=B_{n}$.
If this equality holds, then the cone $A_{n}=B_{n}$ is the nef cone.
\end{proof}
\begin{prop}
For a K3 surface $X$, we can algorithmically decide whether  $\Aut(X)$
is finite.
\end{prop}

\begin{proof}
We define cones $A_{n}$ and $B_{n}$ in $\NS(X)_{\RR}$ as in the
proof of Proposition \ref{prop: compute Nef}. From \cite{kovacs1994thecone},
$A_{n}=B_{n}$ for $n\gg0$ if and only if $\Aut(X)$ is finite. Therefore
we have an algorithm which stops after finitely many steps exactly
when $\Aut(X)$ is finite. We denote this algorithm by $\Theta$. 

For an automorphism $f\colon X\to X$, the tangent space of $\ulAut(X)=\ulIso(X,X)$
at $[f]$ is isomorphic to $\H^{0}(X,\cT_{X})$ with $\cT_{X}$ denoting
the tangent sheaf. Since 
\[
\H^{0}(X,\cT_{X})^{\vee}=\H^{2}(X,\omega_{X}\otimes\Omega_{X})=\H^{2}(X,\Omega_{X})=0,
\]
the Aut scheme $\ulAut(X)$ has only isolated points. From \cite{kondo1999themaximum},
any finite subgroup of $\Aut(X)$ has order at most 3840. The following
algorithm stops after finitely many steps exactly when $\Aut(X)$
is infinite: We enumerate all the Hilbert polynomials as $P_{1},P_{2},\dots$.
\begin{enumerate}
\item Put $n=1$ and $\mathrm{numAuts}=0$. 
\item Put $\mathrm{numAuts}=\mathrm{numAuts}+\sharp\ulIso_{P_{n}}(X,X)$.
\item If $\mathrm{numAuts}>3840$, then stop.
\item Put $n=n+1$ and go to (2). 
\end{enumerate}
We denote this algorithm by $\Theta'$. Now the following algorithm
is the desired one:
\begin{enumerate}
\item Put $n=1$.
\item If $\Theta$ stops after $n$ steps, then stop and output Finite.
\item If $\Theta'$ stops after $n$ steps, then stop and output Infinite.
\item Put $n=n+1$ and go to (2).
\end{enumerate}
\end{proof}
\begin{thm}
\label{thm: decidable K3}For K3 surfaces $X$ and $Y$ with finite
automorphism groups, we can algorithmically decide whether  they are
isomorphic.
\end{thm}

\begin{proof}
We compute the nef cones $\Nef(X)$ and $\Nef(Y)$. There exist at
most finitely many isomorphisms $g\colon\NS(Y)\to\NS(X)$ such that
$g(\Nef(Y))=\Nef(X)$. If there is no such isomorphism, then $X$
and $Y$ are not isomorphic. Suppose that this is not the case and
let $g_{1},\dots,g_{n}$ be all the isomorphisms with this property.
Let $\cL$ and $\cM$ be the given very ample sheaves on $X$ and
$Y$. We compute the Hilbert polynomial for each $[\cL]+g_{i}[\cM]$
and call it by $P_{i}$. Note that the Hilbert polynomial of an ample
invertible sheaf depends only on its numerical class. Indeed, the
numerical class $[\cN]$ of an ample invertible sheaf determines the
Euler characteristics $\chi(\cN)$ and $\chi(\cN^{2})$ from the Riemann-Roch
formula for surfaces. These values together with the one of $\chi(\cO_{X})$
determines the Hilbert polynomial of $\cN$. If there is an isomorphism
$f\colon X\to Y$, then the induced isomorphism $\NS(Y)\to\NS(X)$
is one of the $g_{i}$'s and the Hilbert polynomial of $f$ is one
of the $P_{i}$'s. Thus $[f]$ is a point of $\bigcup_{i=1}^{n}\ulIso_{P_{i}}(X,Y)$.
Thus, $X$ and $Y$ are isomorphic if and only if $\bigcup_{i=1}^{n}\ulIso_{P_{i}}(X,Y)\ne\emptyset$.
From Section \ref{sec:Hom and Iso schemes}, the last condition can
be algorithmically checked. 
\end{proof}
\begin{rem}
For a general K3 surface $X$, there are only finitely many very ample
class $x$ with $x^{2}$ being the prescribed number modulo the action
of $\Aut(X)$ \cite[2.6]{sterk1985finiteness}. Therefore, if we replace
the given very ample sheaf of $X$ by a suitable automorphism of $X$,
we can find an isomorphism $X\to Y$ (if any) with the Hilbert polynomial
in a finite set of potential candidates. But there is a priori no
way to know which automorphism does this job. 
\end{rem}

\begin{rem}
Let $X$ and $Y$ be K3 surfaces and let $\cL$ be the given very
ample sheaf of $Y$. There is no intrinsic invariants of $X$ and
$Y$ to determine the place of $f^{*}[\cL]$ in the ample cone of
$X$ for a potential isomorphism $f\colon X\to Y$. Indeed, when $\rho=2$
and they have infinite automorphisms, then a very ample class $l$
with $l^{2}=d$ is sent to infinitely many distinct lattice points
on the curve $x^{2}=d$. For two lattice points $l_{1}$ and $l_{2}$
on the curve, the sum $l_{1}+l_{2}$ can have arbitrarily large Euler
characteristic. 
\end{rem}

\begin{rem}
Discussion in this section indicates that the isomorphism problem
is closely related to complexity of the automorphism group. Recently,
Lesieutre \cite{lesieutre2018aprojective} showed that there exists
a projective variety $X$ whose automorphism group is discrete, but
not finitely generated (see also \cite{dinh2019asurface}). This result
may be considered to suggest that the isomorphism problem for general
projective schemes is not decidable. 
\end{rem}

\bibliographystyle{alpha}
\bibliography{IsomProb}

\end{document}

%% file: macros.tex
\global\long\def\bigmid{\mathrel{}\middle|\mathrel{}}%

\global\long\def\AA{\mathbb{A}}%

\global\long\def\CC{\mathbb{C}}%

\global\long\def\FF{\mathbb{F}}%

\global\long\def\GG{\mathbb{G}}%

\global\long\def\LL{\mathbb{L}}%

\global\long\def\MM{\mathbb{M}}%

\global\long\def\NN{\mathbb{N}}%

\global\long\def\PP{\mathbb{P}}%

\global\long\def\QQ{\mathbb{Q}}%

\global\long\def\RR{\mathbb{R}}%

\global\long\def\SS{\mathbb{S}}%

\global\long\def\ZZ{\mathbb{Z}}%

\global\long\def\bA{\mathbf{A}}%

\global\long\def\ba{\mathbf{a}}%

\global\long\def\bb{\mathbf{b}}%

\global\long\def\bd{\mathbf{d}}%

\global\long\def\bf{\mathbf{f}}%

\global\long\def\bg{\mathbf{g}}%

\global\long\def\bh{\mathbf{h}}%

\global\long\def\bj{\mathbf{j}}%

\global\long\def\bm{\mathbf{m}}%

\global\long\def\bp{\mathbf{p}}%

\global\long\def\bq{\mathbf{q}}%

\global\long\def\br{\mathbf{r}}%

\global\long\def\bs{\mathbf{s}}%

\global\long\def\bt{\mathbf{t}}%

\global\long\def\bv{\mathbf{v}}%

\global\long\def\bw{\mathbf{w}}%

\global\long\def\bx{\boldsymbol{x}}%

\global\long\def\by{\boldsymbol{y}}%

\global\long\def\bz{\mathbf{z}}%

\global\long\def\bA{\mathbf{A}}%

\global\long\def\bB{\mathbf{B}}%

\global\long\def\bC{\mathbf{C}}%

\global\long\def\bD{\mathbf{D}}%

\global\long\def\bE{\mathbf{E}}%

\global\long\def\bF{\mathbf{F}}%

\global\long\def\bG{\mathbf{G}}%

\global\long\def\bM{\mathbf{M}}%

\global\long\def\bP{\mathbf{P}}%

\global\long\def\bS{\mathbf{S}}%

\global\long\def\bU{\mathbf{U}}%

\global\long\def\bV{\mathbf{V}}%

\global\long\def\bW{\mathbf{W}}%

\global\long\def\bX{\mathbf{X}}%

\global\long\def\bY{\mathbf{Y}}%

\global\long\def\bZ{\mathbf{Z}}%

\global\long\def\cA{\mathcal{A}}%

\global\long\def\cB{\mathcal{B}}%

\global\long\def\cC{\mathcal{C}}%

\global\long\def\cD{\mathcal{D}}%

\global\long\def\cE{\mathcal{E}}%

\global\long\def\cF{\mathcal{F}}%

\global\long\def\cG{\mathcal{G}}%

\global\long\def\cH{\mathcal{H}}%

\global\long\def\cI{\mathcal{I}}%

\global\long\def\cJ{\mathcal{J}}%

\global\long\def\cK{\mathcal{K}}%

\global\long\def\cL{\mathcal{L}}%

\global\long\def\cM{\mathcal{M}}%

\global\long\def\cN{\mathcal{N}}%

\global\long\def\cO{\mathcal{O}}%

\global\long\def\cP{\mathcal{P}}%

\global\long\def\cQ{\mathcal{Q}}%

\global\long\def\cR{\mathcal{R}}%

\global\long\def\cS{\mathcal{S}}%

\global\long\def\cT{\mathcal{T}}%

\global\long\def\cU{\mathcal{U}}%

\global\long\def\cV{\mathcal{V}}%

\global\long\def\cW{\mathcal{W}}%

\global\long\def\cX{\mathcal{X}}%

\global\long\def\cY{\mathcal{Y}}%

\global\long\def\cZ{\mathcal{Z}}%

\global\long\def\fa{\mathfrak{a}}%

\global\long\def\fb{\mathfrak{b}}%

\global\long\def\fc{\mathfrak{c}}%

\global\long\def\ff{\mathfrak{f}}%

\global\long\def\fj{\mathfrak{j}}%

\global\long\def\fm{\mathfrak{m}}%

\global\long\def\fp{\mathfrak{p}}%

\global\long\def\fs{\mathfrak{s}}%

\global\long\def\ft{\mathfrak{t}}%

\global\long\def\fx{\mathfrak{x}}%

\global\long\def\fv{\mathfrak{v}}%

\global\long\def\fD{\mathfrak{D}}%

\global\long\def\fJ{\mathfrak{J}}%

\global\long\def\fG{\mathfrak{G}}%

\global\long\def\fM{\mathfrak{M}}%

\global\long\def\fO{\mathfrak{O}}%

\global\long\def\fS{\mathfrak{S}}%

\global\long\def\fV{\mathfrak{V}}%

\global\long\def\fX{\mathfrak{X}}%

\global\long\def\fY{\mathfrak{Y}}%

\global\long\def\ru{\mathrm{u}}%

\global\long\def\rv{\mathbf{\mathrm{v}}}%

\global\long\def\rw{\mathrm{w}}%

\global\long\def\rx{\mathrm{x}}%

\global\long\def\ry{\mathrm{y}}%

\global\long\def\rz{\mathrm{z}}%

\global\long\def\a{\mathrm{a}}%

\global\long\def\AdGp{\mathrm{AdGp}}%

\global\long\def\Aff{\mathbf{Aff}}%

\global\long\def\Alg{\mathbf{Alg}}%

\global\long\def\age{\operatorname{age}}%

\global\long\def\Ann{\mathrm{Ann}}%

\global\long\def\Aut{\operatorname{Aut}}%

\global\long\def\B{\operatorname{\mathrm{B}}}%

\global\long\def\Bl{\mathrm{Bl}}%

\global\long\def\c{\mathrm{c}}%

\global\long\def\C{\operatorname{\mathrm{C}}}%

\global\long\def\calm{\mathrm{calm}}%

\global\long\def\center{\mathrm{center}}%

\global\long\def\characteristic{\operatorname{char}}%

\global\long\def\codim{\operatorname{codim}}%

\global\long\def\Coker{\mathrm{Coker}}%

\global\long\def\Conj{\operatorname{Conj}}%

\global\long\def\D{\mathrm{D}}%

\global\long\def\Df{\mathrm{Df}}%

\global\long\def\diag{\mathrm{diag}}%

\global\long\def\det{\operatorname{det}}%

\global\long\def\discrep#1{\mathrm{discrep}\left(#1\right)}%

\global\long\def\doubleslash{\sslash}%

\global\long\def\E{\operatorname{E}}%

\global\long\def\Emb{\operatorname{Emb}}%

\global\long\def\et{\textrm{ét}}%

\global\long\def\etop{\mathrm{e}_{\mathrm{top}}}%

\global\long\def\el{\mathrm{e}_{l}}%

\global\long\def\Exc{\mathrm{Exc}}%

\global\long\def\Fitt{\operatorname{Fitt}}%

\global\long\def\Gal{\operatorname{Gal}}%

\global\long\def\GalGps{\mathrm{GalGps}}%

\global\long\def\GL{\mathrm{GL}}%

\global\long\def\Grass{\mathrm{Grass}}%

\global\long\def\H{\operatorname{\mathrm{H}}}%

\global\long\def\hattimes{\hat{\times}}%

\global\long\def\hatotimes{\hat{\otimes}}%

\global\long\def\Hilb{\mathrm{Hilb}}%

\global\long\def\Hodge{\mathrm{Hodge}}%

\global\long\def\Hom{\operatorname{Hom}}%

\global\long\def\hyphen{\textrm{-}}%

\global\long\def\I{\operatorname{\mathrm{I}}}%

\global\long\def\id{\mathrm{id}}%

\global\long\def\Image{\operatorname{\mathrm{Im}}}%

\global\long\def\injlim{\varinjlim}%

\global\long\def\iper{\mathrm{iper}}%

\global\long\def\Iso{\operatorname{Iso}}%

\global\long\def\isoto{\xrightarrow{\sim}}%

\global\long\def\J{\operatorname{\mathrm{J}}}%

\global\long\def\Jac{\mathrm{Jac}}%

\global\long\def\Ker{\operatorname{Ker}}%

\global\long\def\Kzero{\operatorname{K_{0}}}%

\global\long\def\lcr{\mathrm{lcr}}%

\global\long\def\lcm{\operatorname{\mathrm{lcm}}}%

\global\long\def\length{\operatorname{\mathrm{length}}}%

\global\long\def\M{\operatorname{\mathrm{M}}}%

\global\long\def\MHS{\mathbf{MHS}}%

\global\long\def\mld{\mathrm{mld}}%

\global\long\def\mod#1{\pmod{#1}}%

\global\long\def\mRep{\mathbf{mRep}}%

\global\long\def\mult{\mathrm{mult}}%

\global\long\def\N{\operatorname{\mathrm{N}}}%

\global\long\def\Nef{\mathrm{Nef}}%

\global\long\def\nor{\mathrm{nor}}%

\global\long\def\NS{\mathrm{NS}}%

\global\long\def\op{\mathrm{op}}%

\global\long\def\ord{\operatorname{ord}}%

\global\long\def\P{\operatorname{P}}%

\global\long\def\PEff{\mathrm{PEff}}%

\global\long\def\PGL{\mathrm{PGL}}%

\global\long\def\pt{\mathbf{pt}}%

\global\long\def\pur{\mathrm{pur}}%

\global\long\def\perf{\mathrm{perf}}%

\global\long\def\pr{\mathrm{pr}}%

\global\long\def\Proj{\operatorname{Proj}}%

\global\long\def\projlim{\varprojlim}%

\global\long\def\Qbar{\overline{\QQ}}%

\global\long\def\R{\operatorname{\mathrm{R}}}%

\global\long\def\Ram{\operatorname{\mathrm{Ram}}}%

\global\long\def\rank{\operatorname{\mathrm{rank}}}%

\global\long\def\rig{\mathrm{rig}}%

\global\long\def\red{\mathrm{red}}%

\global\long\def\reg{\mathrm{reg}}%

\global\long\def\rep{\mathrm{rep}}%

\global\long\def\Rep{\mathbf{Rep}}%

\global\long\def\sbrats{\llbracket s\rrbracket}%

\global\long\def\Sch{\mathbf{Sch}}%

\global\long\def\sep{\mathrm{sep}}%

\global\long\def\Set{\mathbf{Set}}%

\global\long\def\sing{\mathrm{sing}}%

\global\long\def\sm{\mathrm{sm}}%

\global\long\def\SL{\mathrm{SL}}%

\global\long\def\Sp{\operatorname{Sp}}%

\global\long\def\Spec{\operatorname{Spec}}%

\global\long\def\Spf{\operatorname{Spf}}%

\global\long\def\ss{\mathrm{ss}}%

\global\long\def\st{\mathrm{st}}%

\global\long\def\Stab{\operatorname{Stab}}%

\global\long\def\Supp{\operatorname{Supp}}%

\global\long\def\spars{\llparenthesis s\rrparenthesis}%

\global\long\def\Sym{\mathrm{Sym}}%

\global\long\def\tame{\mathrm{tame}}%

\global\long\def\tbrats{\llbracket t\rrbracket}%

\global\long\def\top{\mathrm{top}}%

\global\long\def\tors{\mathrm{tors}}%

\global\long\def\tpars{\llparenthesis t\rrparenthesis}%

\global\long\def\Tr{\mathrm{Tr}}%

\global\long\def\ulAut{\operatorname{\underline{Aut}}}%

\global\long\def\ulHom{\operatorname{\underline{Hom}}}%

\global\long\def\ulIso{\operatorname{\underline{{Iso}}}}%

\global\long\def\ulSpec{\operatorname{\underline{{Spec}}}}%

\global\long\def\Utg{\operatorname{Utg}}%

\global\long\def\Unt{\operatorname{Unt}}%

\global\long\def\Var{\mathbf{Var}}%